\documentclass[preprint,nonatbib]{elsarticle}
\makeatletter
\let\c@author\relax
\makeatother
\usepackage[utf8]{inputenc}
\usepackage[USenglish,american]{babel} 
\usepackage{csquotes}

\usepackage{amsmath,amsthm}
\usepackage{amssymb}
\usepackage{mathtools}
\usepackage{tikz}
\usetikzlibrary{backgrounds}
\usetikzlibrary{shapes.geometric, arrows}
\usepackage{todonotes}
\usepackage{xspace}
\usepackage[shortlabels]{enumitem}
\usepackage[nolist]{acronym}
\usepackage{float} 
\usepackage{longtable} 
\usepackage{booktabs} 
\usepackage[table]{xcolor}  
\usepackage{colortbl}       
\usepackage{graphicx} 
\usepackage{adjustbox}
\usepackage{subcaption} 
\usepackage{placeins} 
\usepackage[ruled,englishkw,linesnumbered,nosemicolon]{algorithm2e}
\usepackage{soul} 

\usepackage[hidelinks]{hyperref}
\usepackage[capitalise]{cleveref}

\crefformat{footnote}{#2\footnotemark[#1]#3}

\usepackage[
    backend=biber,
    natbib=true,
    url=false, 
    doi=true,
    eprint=false
]{biblatex}
\newtheorem{lemma}{Lemma}
\newtheorem{definition}{Definition}

\makeatletter
\AtBeginDocument
 {
   \def\ltx@label#1{\cref@label{#1}}
   \def\label@in@display@noarg#1{\cref@old@label@in@display{#1}}
\def\label@in@mmeasure@noarg#1{%
    \begingroup%
      \measuring@false%
      \cref@old@label@in@display{#1}
    \endgroup}%
 } %
\makeatother

\renewcommand{\P}{\mathcal{P}}  
\newcommand{\Ppool}{\ensuremath{\P_{\text{Pool}}}}

\newcommand{\R}{\mathcal{R}}    
\newcommand{\T}{\mathcal{T}}    
\newcommand{\nT}{T}             
\newcommand{\nf}{F}             
\newcommand{\nm}{M}             

\newcommand{\arr}{a}    
\newcommand{\dis}{d}    
\newcommand{\reg}{\nu}  

\newcommand{\los}{\mathrm{los}} 
\newcommand{\croom}{c_r} 
\newcommand{\ctotal}{c_{\R}} 
\newcommand{\loadf}{\ensuremath{\ell}}

\newcommand{\N}{\mathbb{N}}

\newcommand{\Ex}{\mathbb{E}}

\DeclarePairedDelimiter\ceil{\lceil}{\rceil}

\definecolor{dunkelblau2}{RGB}{0,83,159}
\definecolor{dunkelblau}{RGB}{0,113,187}
\definecolor{hellblau2}{RGB}{130,169,208}
\definecolor{hellblau}{RGB}{125,167,217}
\definecolor{gruen}{RGB}{179,210,53}
\definecolor{MyOrange}{RGB}{251,166,28}
\definecolor{MyDunkelgruen}{RGB}{0,102,105}
\definecolor{MyTuerkis}{RGB}{0,170,173}

\tikzset{
	box/.style = { draw = black, very thick, scale=2, minimum width=.75cm}, 
	box1/.style = { draw = black, very thick, scale=2, fill=black!10},
	box2/.style = { draw = black, very thick, scale=2, minimum width=.75cm, fill=black!10},
	greenball/.style = {circle, draw=black,  thick, fill = gruen, scale=.8},
	redball/.style = {circle, draw=black,thick, fill = red!80, scale=.8 },
	blacknode/.style = {circle, draw=black,  thick, fill=black!100, scale=0.35},
	blacknode2/.style = {circle, draw=black, very thick, fill=black!0, scale=0.7},
	darkgreennode/.style = {circle, draw=black, very thick, fill=dunkelgrun, scale=0.7},
	darkbluenode/.style = {circle, draw=black, very thick, fill=dunkelblau, scale=0.7},
	greennode/.style = {circle, draw=black, very thick, fill=grun, scale=0.7},
	rednode/.style = {circle, draw=black, very thick, fill=red!80, scale=0.7},
	bluenode/.style = {circle, draw=black, very thick, fill=hellblau!80, scale=0.7},
	orangenode/.style = {circle, draw=black, very thick, fill=orange, scale=0.7},
	cross/.style = {cross out, draw=black, thick, fill=black!10, scale=0.6},
	edge/.style = { thick},
	edge2/.style = {dashed,  very thick, draw=gray},
	grayball/.style = {draw=gray, fill=gray!20, thick, circle, scale=.8},
	orangeball/.style = {draw=black, fill=orange, thick, circle},
	square/.style = { draw=none, fill=dunkelblau!10},
	blacksquare/.style = { draw=black, very thick, fill=none}
}

\begin{document}
\begin{acronym}[ECU] 
    \acro{los}[LOS]{length of stay}
    \acro{lor}[LOR]{length of registration}
    \acro{ssp}[SSP]{subset sum problem}
    \acro{pra}[PRA]{patient-to-room assignment problem}
    \acro{bkp}[BSKP]{bounded simple knapsack problem}
    \acro{gui}[GUI]{graphical user interface}
\end{acronym}

\begin{frontmatter}
\title{
Instance Generation for Patient-to-room Assignment and Admission Scheduling
Based on Real Hospital Data
}
\author[acomb]{Tabea Brandt\corref{cor1}\fnref{fn1}}
\ead{brandt@combi.rwth-aachen.de}
\author[acomb]{Christina B\"using\corref{cor1}\fnref{fn1}} 
\ead{buesing@combi.rwth-aachen.de}
\author[acomb]{Johanna Leweke\corref{cor1}\fnref{fn1}} 
\ead{leweke@combi.rwth-aachen.de}
\author[aalg]{Finn Seesemann\corref{cor1}\fnref{fn1}}
\ead{seesemann@algo.rwth-aachen.de}
\author{Sina Weber} 

\address[acomb]{Chair of Combinatorial Optimization, RWTH Aachen University, \\
Im Süsterfeld 9, 52072 Aachen, Germany}
\address[aalg]{Chair of Algorithms and Complexity, RWTH Aachen University, \\
Ahornstraße 55, 52074 Aachen, Germany}
\fntext[1]{
This work was supported
by the German research council (DFG) Research Training Group 2236 UnRAVeL.
Declarations of interest: none
}
\cortext[cor1]{Corresponding author}

\journal{arXiv.org}
\begin{abstract}
Developing algorithms for real-life problems that perform well in practice depends on the availability of realistic data for testing.
Obtaining real-life data for optimization problems in health care, however, is often difficult, and such data typically cannot be published, which limits reproducibility by other researchers. This is especially true for patient-related problems because of data privacy policies such as the \acf{pra}.
Therefore, artificially generated instances are commonly used. To improve the generation of realistic instances, we develop a configurable instance generator for \ac{pra} and other patient-related problems, featuring an easy-to-use graphical user interface.
The design of the generator is based on an extensive empirical analysis of real hospital data, which identifies relevant ward-specific patterns such as patients’ age and length-of-stay distributions.
Moreover, as randomly generated instances are often infeasible, we address this issue in two ways. We implement a dynamic programming approach in the generator to optionally enforce feasibility and extend existing results from the literature to derive new combinatorial insights into \ac{pra} feasibility.

\end{abstract}
\end{frontmatter}

\section{Introduction}\label{sec:introduction}
The development of well-performing algorithms for real-life optimization problems depends on the availability of realistic input data \cite{Lent2012}.
In the health care sector, however, obtaining real data for testing is often challenging due to privacy regulations, incomplete documentation, and limited data accessibility \cite{Harper2004, Brailsford2013}.
This challenge is especially pronounced for patient-related problems, such as the \acf{pra}, which crucially rely on detailed patient information \cite{Demeester2010, Brailsford2010}.
Even when real data can be obtained, it typically cannot be published due to confidentiality concerns, which complicates the reproducibility and comparability of computational studies.
Consequently, difficulties related to data collection, accessibility, and research ethics contribute to the perception that modeling in health care is comparatively difficult \cite{Tako2015}. 

To compensate for the lack of real data, artificial instances are frequently generated using small samples or expert estimates. However, such data often fails to adequately capture the complexity and heterogeneity of hospital environments and may therefore lead to unreliable conclusions \cite{TunnicliffeWilson1981}.
For example, expert estimates are often affected by systematic biases, as experts tend to overweight exceptional or high-impact situations \cite{Carter2023}.
For the \ac{pra}, we further identify significant differences between hospital wards, for example with respect to patients' age distributions and \acl{los}, making data derived from a single ward not transferable to other wards and thus other research projects. 


To address this gap, we propose a new instance generator for patient data with the \ac{pra} in mind. 
Our approach is based on an extensive data analysis on a large real-life hospital dataset.
In the analysis we identify systematic differences in patient attributes across wards and cluster the wards into distinct types for each relevant attribute.
Based on these clusters, we derive explicit probability distributions. The generator provides these distributions through an easy-to-use graphical user interface which is accessible in form of a web application. 
Users can configure key characteristics of the generated instances, including patient demographics (e.g., age and gender), admission characteristics (e.g., emergency rate), stay-related parameters (e.g., length of stay), and the ward layout (e.g., the number and capacities of rooms), to represent typical ward types or to mimic specific hospital settings.
A distinctive feature of the generator is the optional enforcement of feasibility with respect to gender-separated rooms.

The generated datasets correspond to the data required in the \ac{pra}, but can also serve as input for other patient-related planning problems in hospitals. In particular, the instances support capacity planning decisions, such as determining the required number of beds in a ward or assessing the impact of temporarily closing beds. 
Moreover, they can be used for patient admission management, including decisions on postponing or canceling elective patients under capacity constraints. Beyond these applications, the datasets can be extended to address additional problems, for example by incorporating patient workload for nurse-to-patient assignment problems.

The remainder of this paper is organized as follows:
\cref{sec:problem_def} formally introduces the \ac{pra} and the corresponding instance structure considered in this work.
\cref{sec:literature} provides an overview of existing instance sets and generators for the \ac{pra}.
\cref{sec:generator} analyzes the real-life hospital data underlying the instance generator and describes its input parameters, usage options, and the generation process.
\cref{sec:feas} addresses the feasibility of instances with respect to gender-separated rooms. 
In this context, we show how the dynamic programming algorithm for the knapsack problem can be applied to arbitrary instances of the feasibility problem and identify polynomially solvable special cases.
Finally, \cref{sec:future} concludes the paper and outlines directions for future research.



\section{Problem definition} \label{sec:problem_def}
The instance generator produces exactly the parameters required to define an instance of the \acf{pra}.
In the \ac{pra}, we are given a hospital ward with a set of rooms $\R$ with $\croom \in \N$ beds in room $r \in \R$, and a discrete planning horizon $\mathcal{T} = \{ 1, \dots, T \}$. The total bed capacity is denoted by $\ctotal = \sum_{r \in \R} \croom$.
Furthermore, we have a set of patients $\P$, where $\P(t) \subseteq \P$ denotes the subset of patients present in the ward on day $t \in \mathcal{T}$.
For each patient $p \in \P$ we know their age, sex, registration date $\reg(p) \in \mathcal{T}$, arrival date $\arr(p) \in \mathcal{T}$, discharge date $\dis(p) \in \mathcal{T}$, whether they are entitled to a single room, and whether they bring an accompanying person.
The \ac{los} is defined as $\los(p) = \dis(p) - \arr(p)$. Under the assumption that $1 \leq \arr(p) \leq \dis(p) \leq T$ holds for all patients, it follows that
\begin{align} \label{eq:los_to_patientCount}
    \sum_{p \in \P} \los(p) = \sum_{t \in \mathcal{T}} |\P(t)|.
\end{align}
Analogously, the \ac{lor} is defined as $\arr(p) - \reg(p)$.
The ward utilization on day $t \in \mathcal{T}$ is quantified by the daily load factor $\loadf_t = |\mathcal{P}(t)| / \ctotal$, and the overall load factor over the planning horizon is $\loadf = \left( \sum_{t \in \mathcal{T}} \loadf_t \right) / T$.

The task in the \ac{pra} is to assign every patient to a room for every time period of their stay.
We assume that patients stay in the hospital on consecutive time periods from arrival to discharge and that they are discharged at the beginning of a time period. Furthermore, we assume that every patient can be assigned to any room. An assignment of patients to rooms is feasible if, for every room and every time period, the following two conditions are fulfilled: (i) rooms capacities are respected, and (ii) female and male patients are assigned to separate rooms.
Under these assumptions, single-room entitlement is treated as a soft constraint; therefore, it does not affect the feasibility of the model.

Checking feasibility of an instance of the \ac{pra} is $\mathcal{NP}$-complete in general \cite{brandt_2024}, but some special cases can be solved in polynomial time, as we show in \cref{sec:feas}. We also present a pseudo-polynomial dynamic program for arbitrary room capacities.
We integrated the dynamic program into the instance generator to guarantee the construction of feasible instances even in case of a high load factor $\loadf \in [0,1]$.

\section{Literature}\label{sec:literature}
To our knowledge, there are currently four instance sets providing patient data suitable for \ac{pra}. One of them is provided without an accompanying instance generator, two include publicly available generators, and a fourth requires access to a large underlying dataset to generate additional instances.

The first benchmark instance set for \ac{pra} was proposed by Demeester et al.\ and consists of 13 instances available online\footnote{\url{https://people.cs.kuleuven.be/~tony.wauters/wim.vancroonenburg/pas/}} \cite{Demeester2010}.
The authors consider a centralized bed-management system for an entire hospital, where patients must be assigned to the most suitable room with respect to department specialty and room equipment.
The instances are based on a realistic hospital setting and informed by interviews with experienced admission schedulers.
Each instance models a two-week planning horizon in a hospital with six departments, each comprising 20--30 rooms with different equipment.
Patients' \ac{los} are generated according to a normal distribution with mean 5 and variance 3, and 60 new patients arrive per day.
No probability distributions are provided for additional patient attributes such as sex, age, or \ac{lor}.

Ceschia and Schaerf provide an instance set and generator for the integrated problem of \ac{pra} and operating-room scheduling\footnote{\url{https://bitbucket.org/satt/or-pas/src/master/}} \cite{Ceschia2016}.
The generator requires as input the number of rooms and beds, the number of patients, and the planning horizon.
Patient registrations per day are determined by a Poisson distribution with an expected value equal to the input average number of patients.
Admission dates are not fixed a priori but must be scheduled within a given time window: 10\% of patients are urgent and require immediate admission, 45\% must be admitted within one week, and the remaining 45\% within two weeks.
\ac{los} follow treatment-specific log-normal distribution.
Patients' gender is assigned randomly, age is drawn from a uniform distribution, and 30\% of patients have an overstay risk of one additional night.

Brandt et al.\ present an instance set and generator for the integrated problem of \ac{pra} and nurse-to-patient assignment\footnote{\url{https://doi.org/10.5281/zenodo.12750420}} \cite{Brandt2025}.
We summarize here only the \ac{pra}-relevant aspects.
The generator takes as input the number of instances, the planning horizon, the number of rooms by capacity, the target load factor, and available room equipment types.
The number of generated patients depends on the room configuration and load factor.
Rooms are assigned random subsets of room equipment, and patients are characterized by an admission shift, a \ac{los} uniformly drawn between one and five, a randomly assigned gender, and a subset of required equipment.
For integration with nurse-to-patient assignment, patient-specific workloads are specified for each day and shift.

In addition to these three \ac{pra}-specific generators, a large artificial data set is provided by the NHS\footnote{\url{https://digital.nhs.uk/services/artificial-data}}.
While this data can be used to construct \ac{pra} instances by defining appropriate wards and rooms, this decision is nontrivial.
The data are generated from real NHS data via statistical aggregated occurrences and thus realistically capture global pattern such as weekly admission fluctuation.
However, ward-specific characteristics and dependencies between patient attributes, such as age and \ac{los}, are not necessarily preserved. The generation code is publicly available, allowing users with access to extensive real data to produce anonymized data sets.




All reviewed instance generators rely on fixed probability distributions and generate patient attributes independently.
In contrast, our data analysis shows that key attributes such as the \ac{los} vary substantially across wards and strongly correlates with patient age.
Consequently, our instance generator provides ward-specific and age-dependent probability distributions for attributes such as the \ac{los}.
Moreover, the proposed generator offers the option to explicitly enforce feasibility with respect to gender-separated rooms, a requirement that is mandatory in several countries, including Germany.
Without such enforcement, instances with long planning horizons or high load factors are likely to be infeasible.
Alternatively, feasibility can be disabled, allowing the generation of instances with load factors exceeding one, thereby supporting admission control decisions.

Since the available real-world dataset does not contain information on room equipment, we do not model equipment explicitly.
Instead, all equipment is assumed to be rolling stock and therefore independent of rooms.
While the instances of Demeester et al.\ consider multiple departments, we could not identify suitable data describing alternative wards to which patients may be assigned.
Consequently, the current version of the generator does not support the generation of multi-ward assignment instances.

\section{Generator}\label{sec:generator}
In this section, we present a detailed description of our instance generator. 
We first describe the data used for the generator.
Second, we explain how to use the generator, and third, we describe the technical generation process.
The code for our instance generator is publicly available at \cite{unserCode}.

\subsection{Description and analysis of patient data from the hospital}
For our instance generator, we had access to one year of data from a hospital with more than 60 wards and over 50,000 patients. 
We excluded data from pediatric and psychiatric wards. Pediatric patients differ significantly in age and in the presence of accompanying persons, while psychiatric wards were omitted on the hospital's recommendation, as patient behavior there is not representative and lacks a discernible pattern.
Accordingly, the dataset contains only patients aged 18 years or older. The analytical results and the generated instances are not intended to be representative of pediatric or psychiatric wards.

For some patients, the data reports consecutive stays in multiple wards.
To obtain ward-specific data, we treat each stay per ward as an independent patient.
For each patient, we have the registration, admission, and discharge date, as well as age, gender, single-room entitlement, and the presence of an accompanying person.
When a patient’s stay is split into multiple ward-specific segments, each segment retains the original registration date, and the corresponding \ac{lor} values reflect this shared registration date.

Our data analysis reveals that patient age correlates with almost all other patient attributes.
To account for this, we compute ward-specific age probabilities for ages 18 to 100  and cluster the wards accordingly into four distinct types, cf. \cref{fig:rate_params:age}.
For each type, we derive an empirical age distribution based on the aggregated observations of all wards assigned to that type.
The overall age distribution across the full dataset can be reasonably approximated by a normal distribution  for the purpose of instance generation, with mean $\mu=61.559$ and standard deviation $\sigma=17.496$. These parameters were obtained using a standard fitting procedure\footnote{\url{https://docs.scipy.org/doc/scipy/reference/generated/scipy.stats.norm.html}}.
\begin{figure}[ht!]
    \centering
    \begin{subfigure}[b]{0.49\textwidth}
        \includegraphics[scale=0.35]{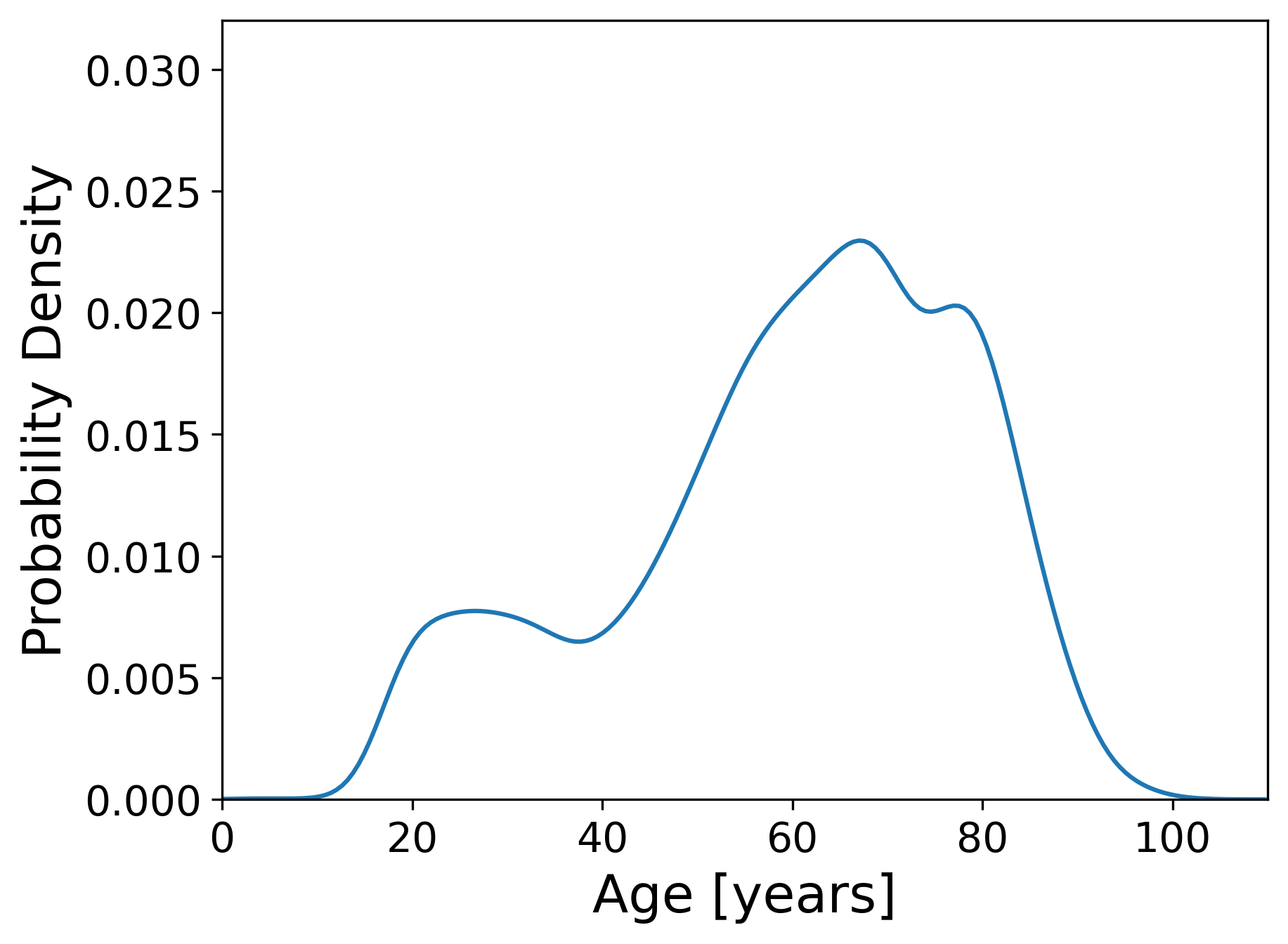}
        \caption{Type 1: 2-step peak}
        \label{fig:rate:age1}
    \end{subfigure}
    \begin{subfigure}[b]{0.49\textwidth}
        \includegraphics[scale=0.35]{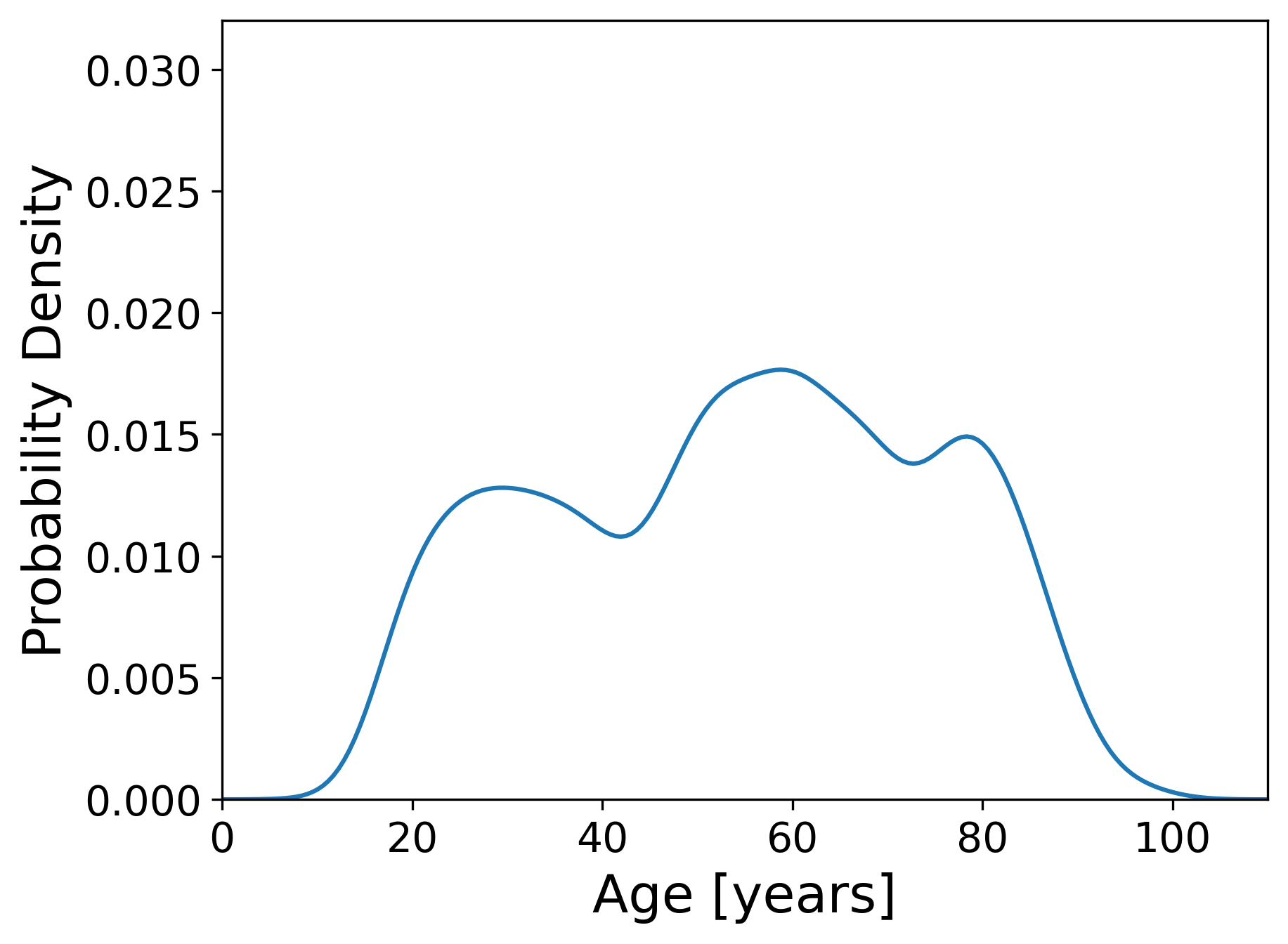}
        \caption{Type 2: 3-step peak}
        \label{fig:rate:age2}
    \end{subfigure}
    \begin{subfigure}[b]{0.49\textwidth}
        \includegraphics[scale=0.35]{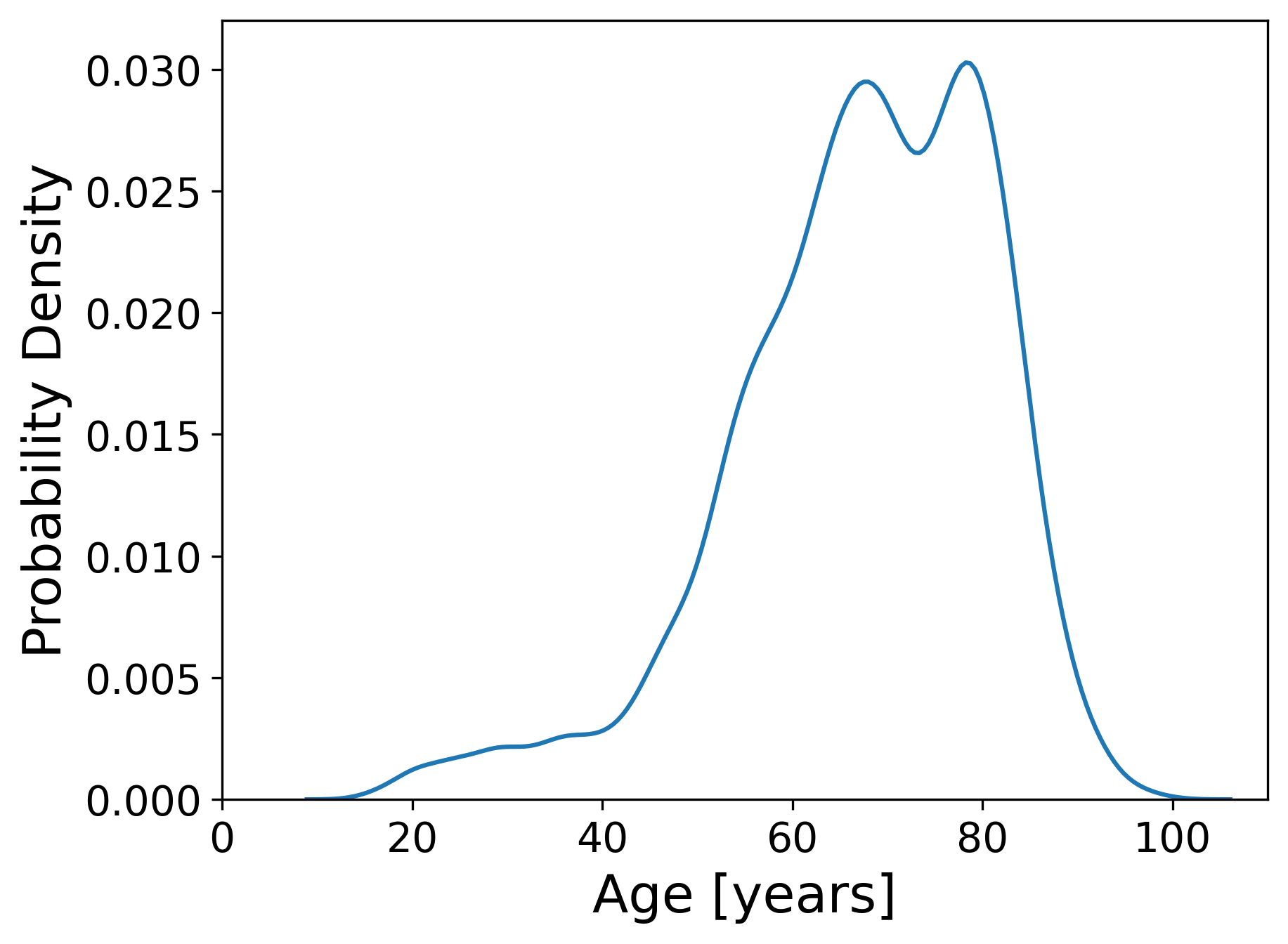}
        \caption{Type 3: high peak}
        \label{fig:rate:age3}
    \end{subfigure}
    \begin{subfigure}[b]{0.49\textwidth}
        \includegraphics[scale=0.35]{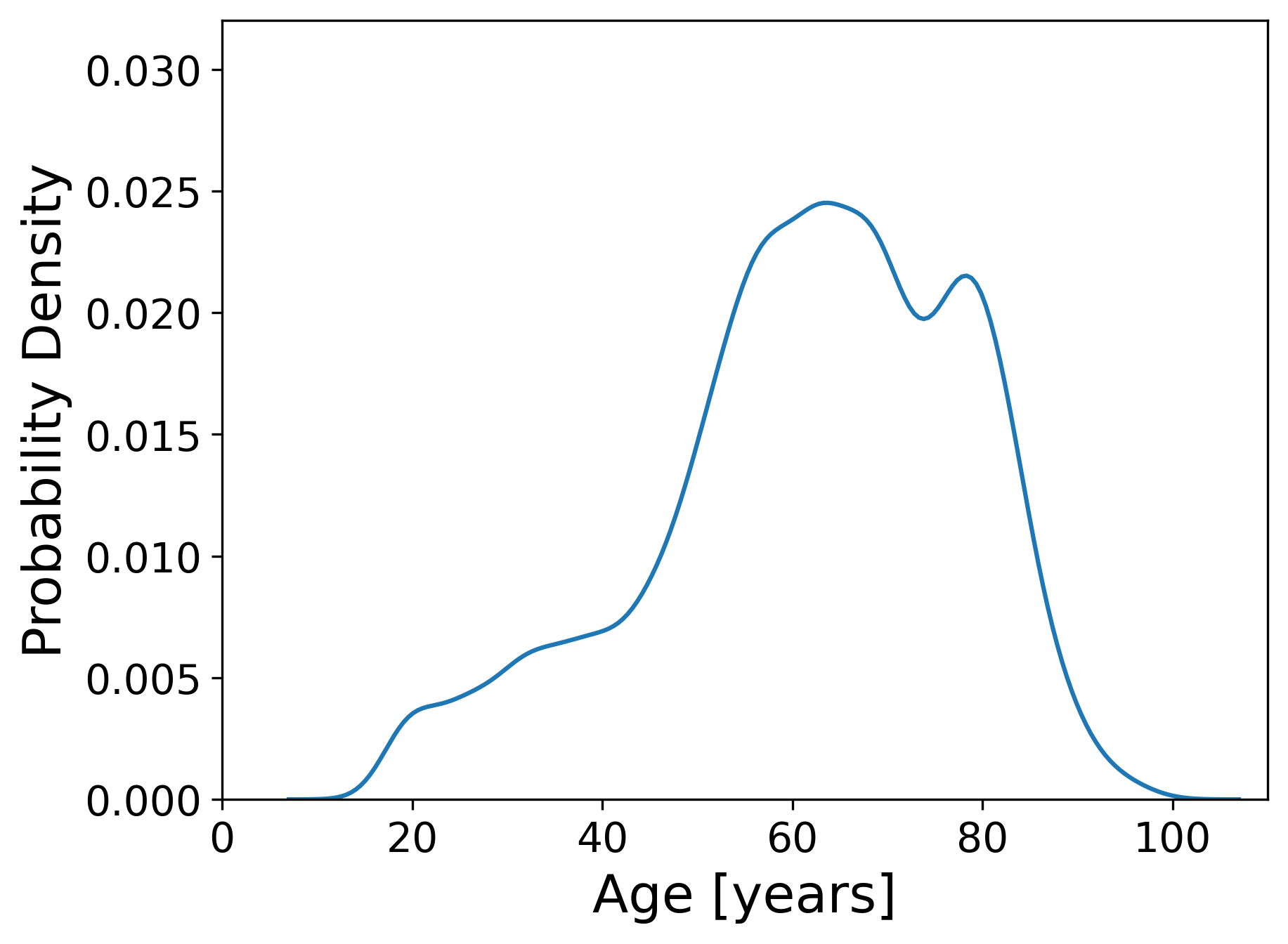}
        \caption{Type 4: wide peak}
        \label{fig:rate:age4}
    \end{subfigure}
    \caption{Four types of age distributions}
    \label{fig:rate_params:age}
\end{figure}

For each ward, we compute the probability distribution of a patient's \ac{los}.
Analogous to the age-based analysis, we cluster the wards according to their \ac{los} distributions (cf. \cref{fig:rate_params:los}), yielding five distinct distribution types. We also derive a global \ac{los} distribution based on the entire dataset.
Consistent with the literature, this global distribution is well approximated by a log-normal model \cite{Marazzi1998, Harper2002}, with mean $\mu=4.021$ and standard deviation $\sigma=1.246$, obtained via numerical fitting\footnote{\label{note:lognormal_fit}\url{https://docs.scipy.org/doc/scipy/reference/generated/scipy.stats.lognorm.html}}.
\begin{figure}[ht!]
    \centering
       \begin{subfigure}[b]{0.49\textwidth}
        \includegraphics[scale=0.35]{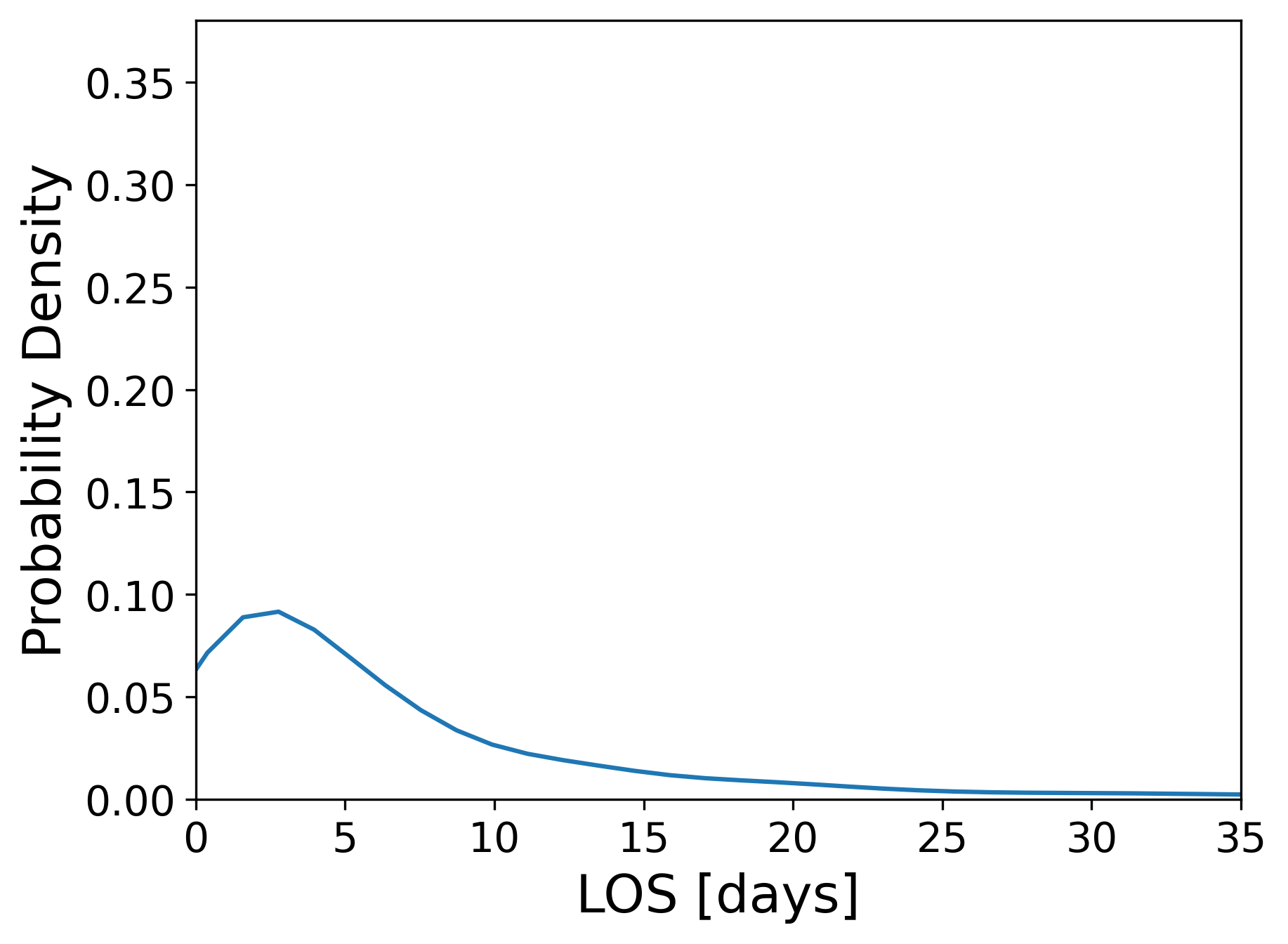}
        \caption{Type 1: flat curve}
        \label{fig:rate:los1}
    \end{subfigure}
    \begin{subfigure}[b]{0.49\textwidth}
        \includegraphics[scale=0.35]{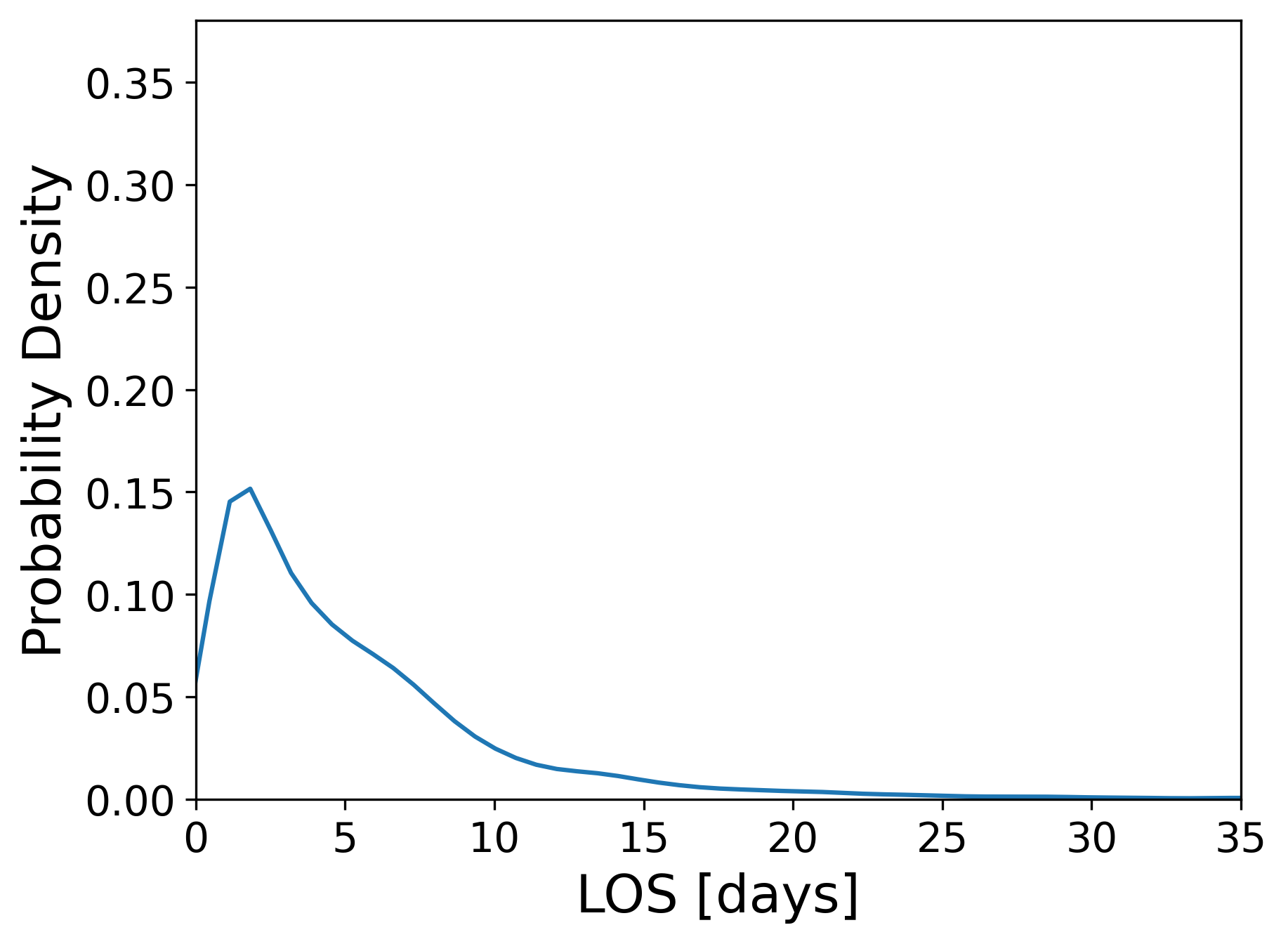}
        \caption{Type 2: middle curve}
        \label{fig:rate:los2}
    \end{subfigure}
       \begin{subfigure}[b]{0.49\textwidth}
        \includegraphics[scale=0.35]{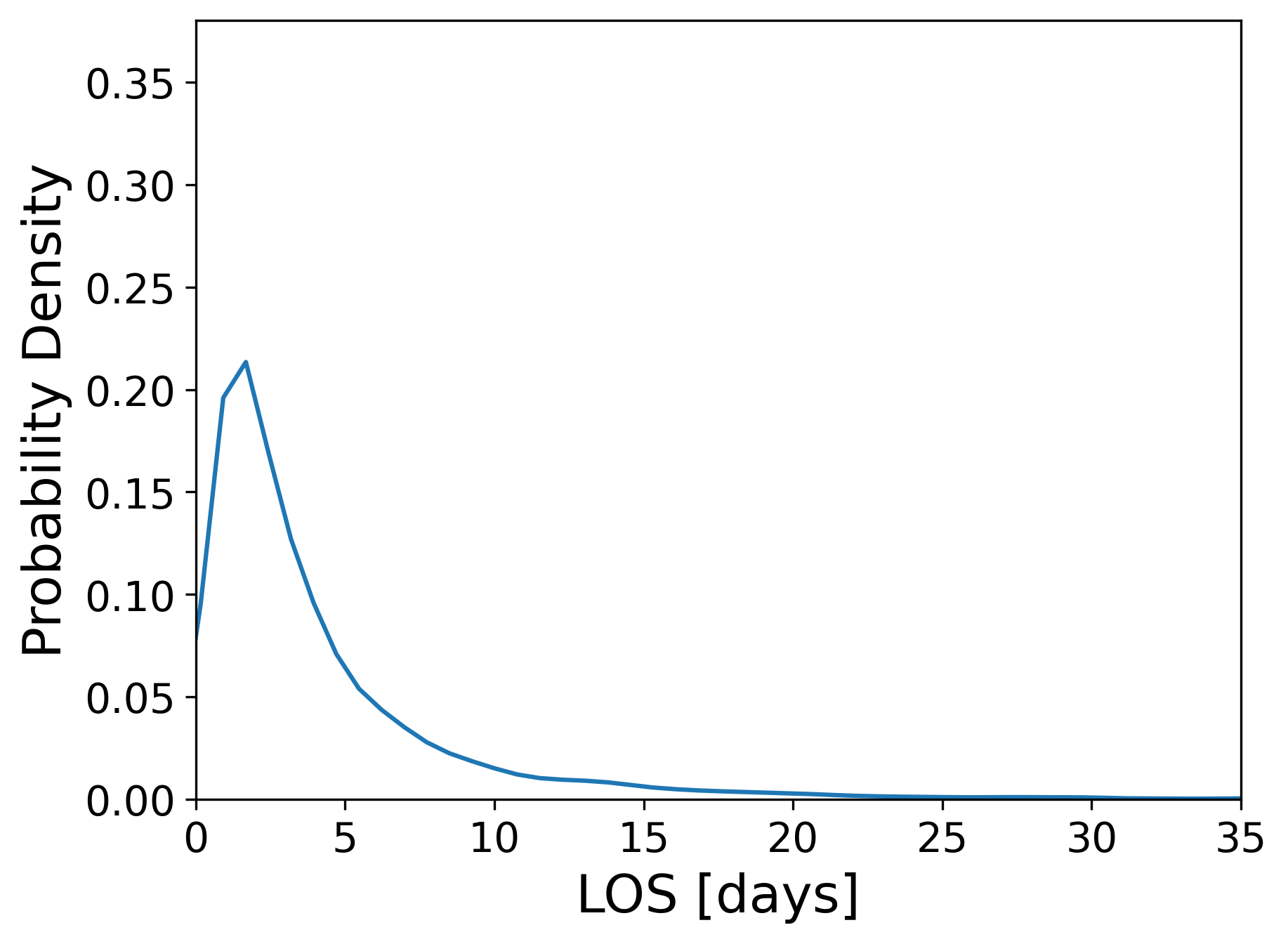}
        \caption{Type 3: higher curve}
        \label{fig:rate:los3}
    \end{subfigure}
    \begin{subfigure}[b]{0.49\textwidth}
        \includegraphics[scale=0.35]{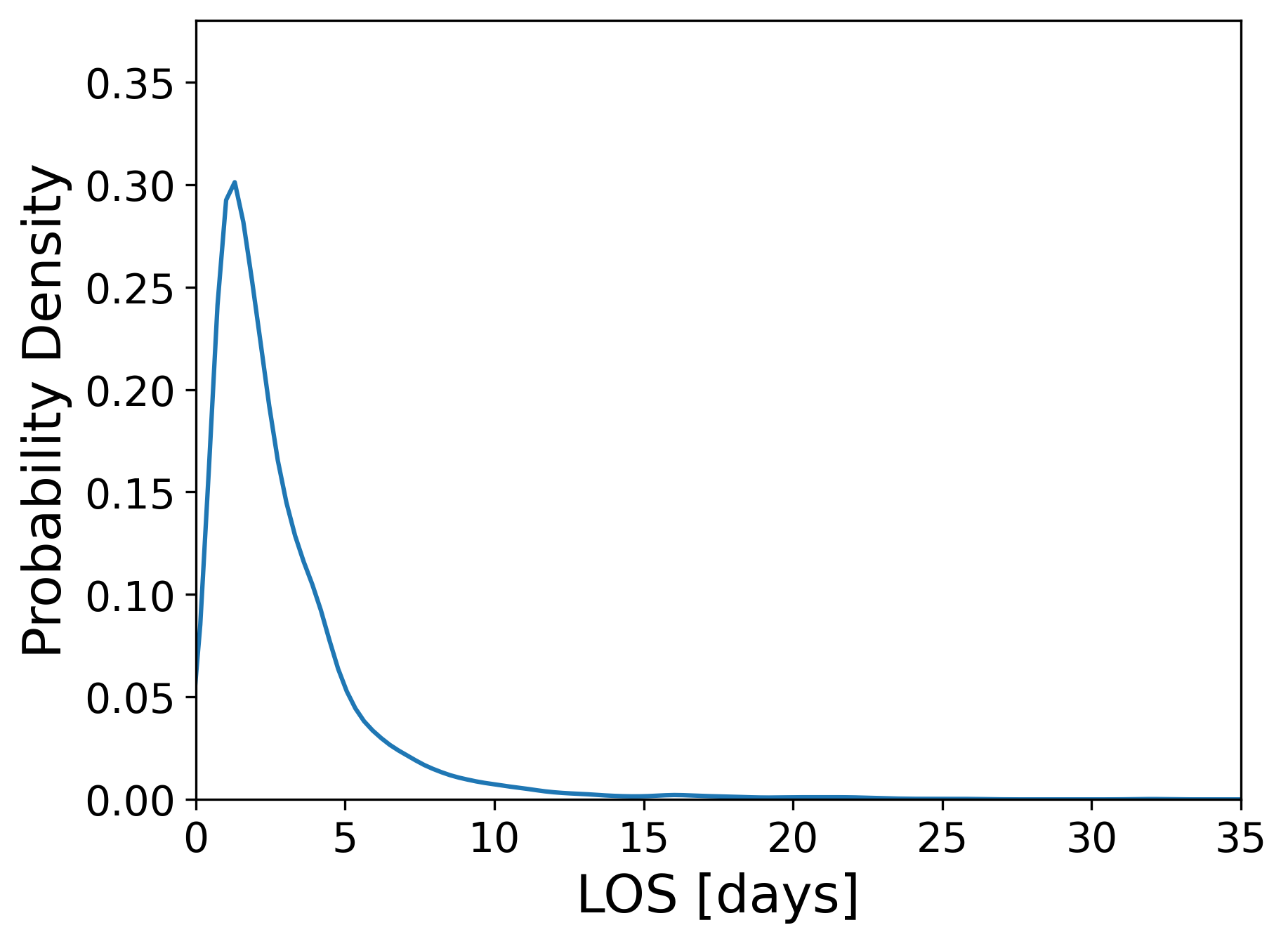}
        \caption{Type 4: very high curve}
        \label{fig:rate:los4}
    \end{subfigure}
    \begin{subfigure}[b]{0.49\textwidth}
        \includegraphics[scale=0.35]{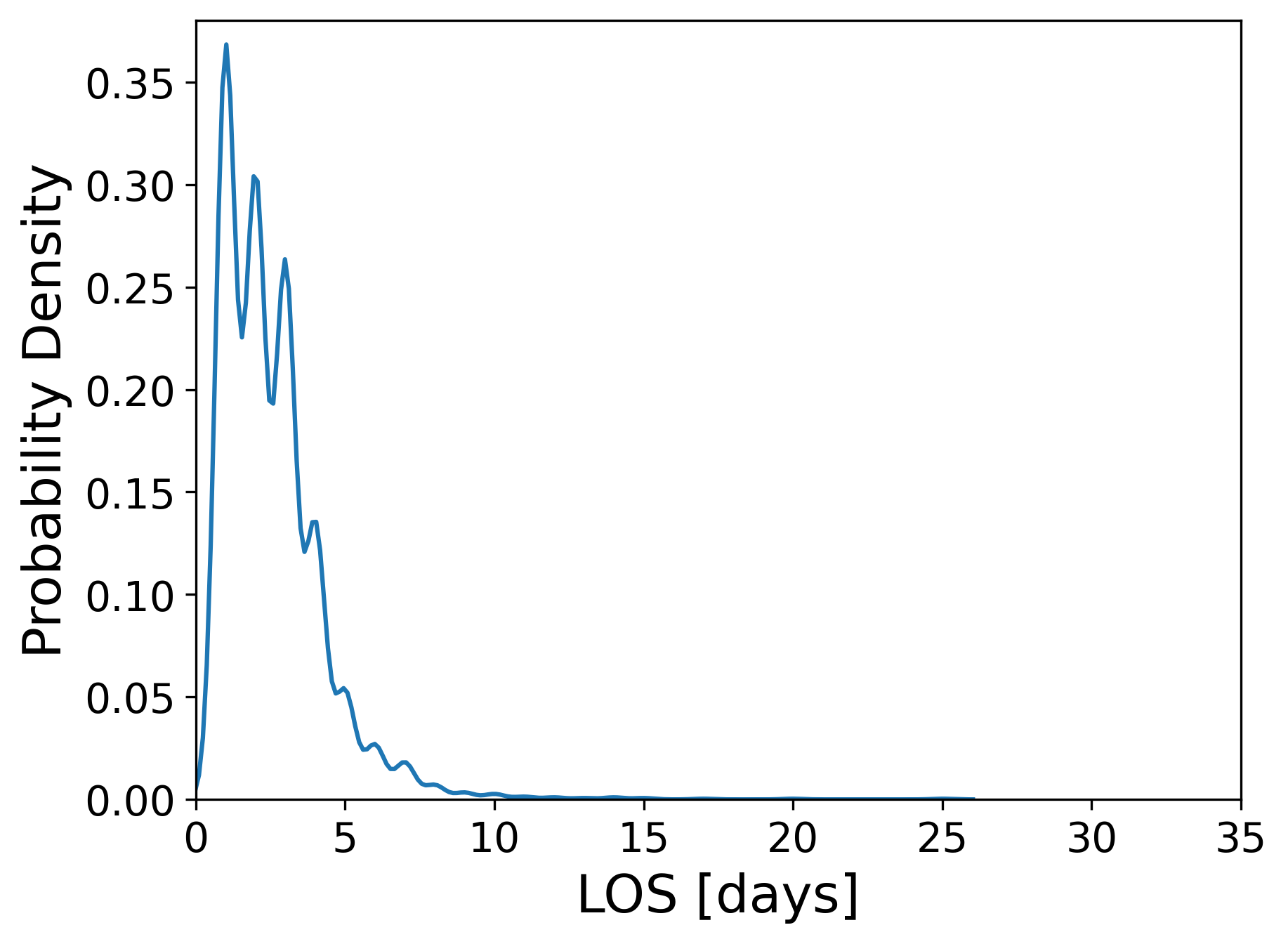}
        \caption{Type 5: extreme curve}
        \label{fig:rate:los5}
    \end{subfigure}
    \caption{Five types of \ac{los} distributions}
    \label{fig:rate_params:los}
\end{figure}

We further observe a strong correlation between a patient's age and \ac{los}.
To capture this relationship, we compute a joint probability distribution for each ward.
Patients within a ward are grouped into age classes of width 5 to ensure sufficient data per class for estimating the \ac{los} distribution.
This yields a three-dimensional representation assigning a probability to each combination of age class and \ac{los}, where \ac{los} is restricted to the range 0--24 days, which covers the vast majority of observed patients.
Similarities in these joint distributions across wards allow us to cluster them into five types. We visualize these types using heat maps, where darker regions indicate higher probability, cf. \cref{fig:rate_params:agelos}.
\begin{figure}[ht!]
    \centering
       \begin{subfigure}[b]{0.49\textwidth}
        \includegraphics[scale=0.35]{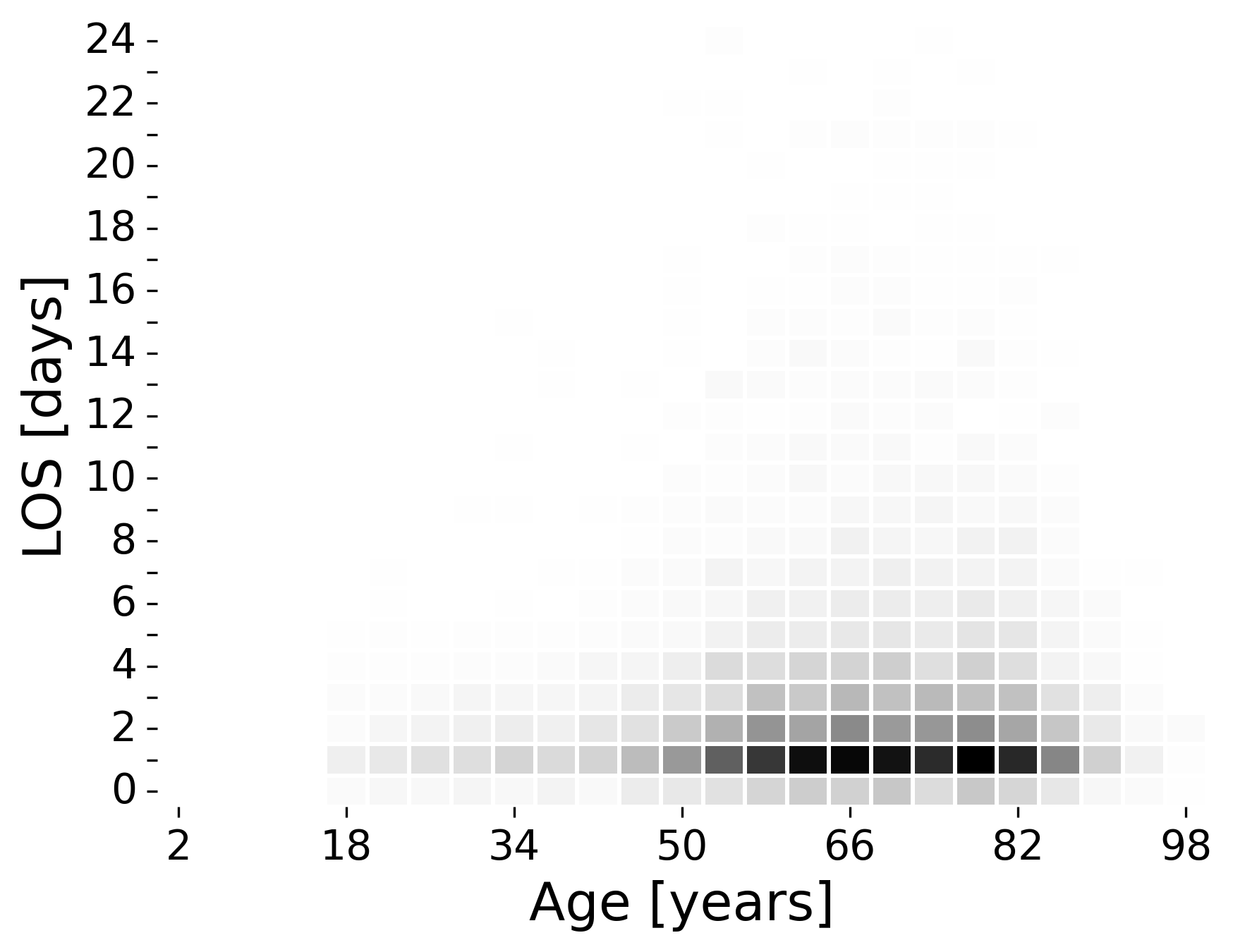}
        \caption{Type 1: short}
        \label{fig:rate:agelos1}
    \end{subfigure}
    \begin{subfigure}[b]{0.49\textwidth}
        \includegraphics[scale=0.35]{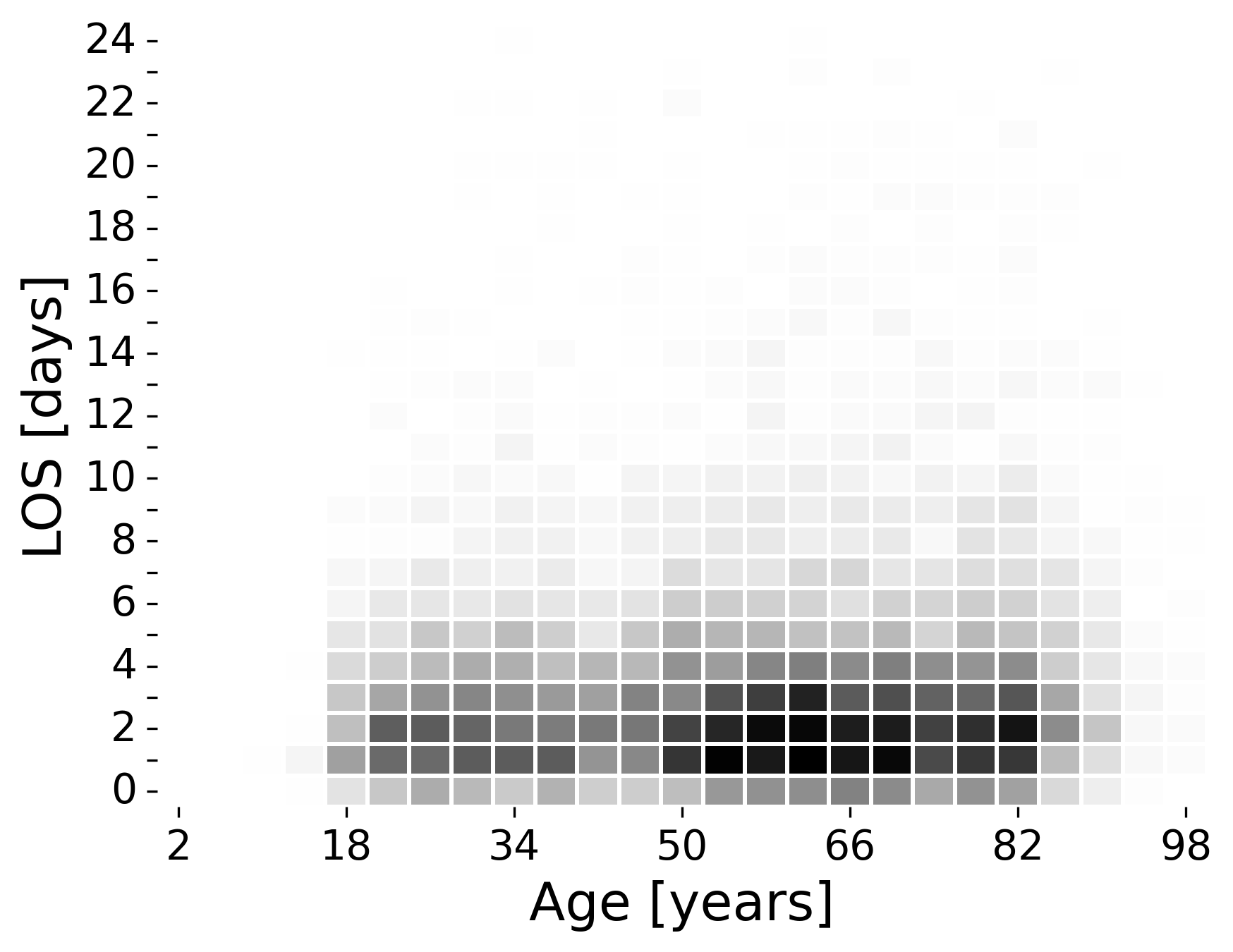}
        \caption{Type 2: middle}
        \label{fig:rate:agelos2}
    \end{subfigure}
       \begin{subfigure}[b]{0.49\textwidth}
        \includegraphics[scale=0.35]{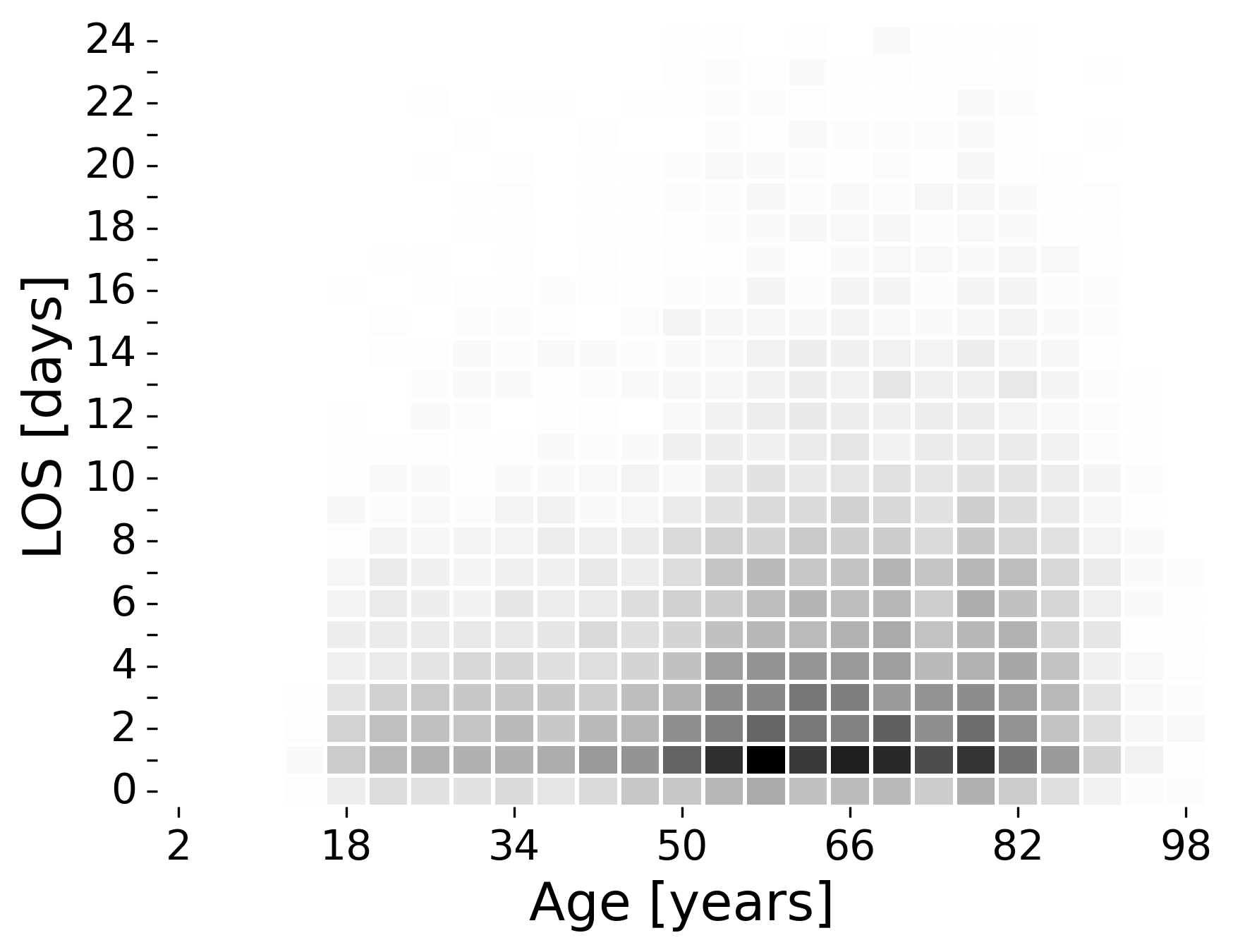}
        \caption{Type 3: triangle}
        \label{fig:rate:agelos3}
    \end{subfigure}
    \begin{subfigure}[b]{0.49\textwidth}
        \includegraphics[scale=0.35]{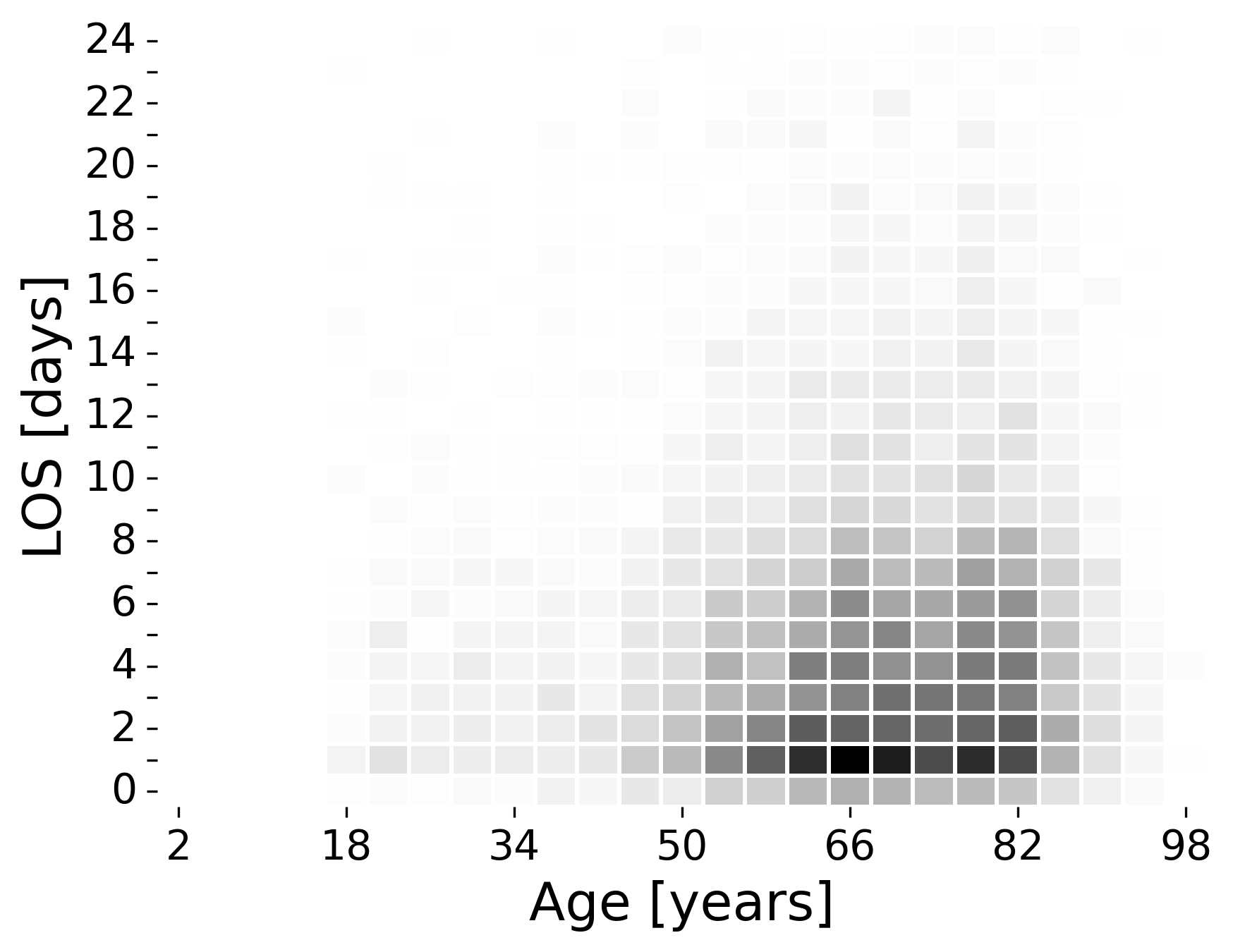}
        \caption{Type 4: short and old}
        \label{fig:rate:agelos4}
    \end{subfigure}
    \begin{subfigure}[b]{0.49\textwidth}
        \includegraphics[scale=0.35]{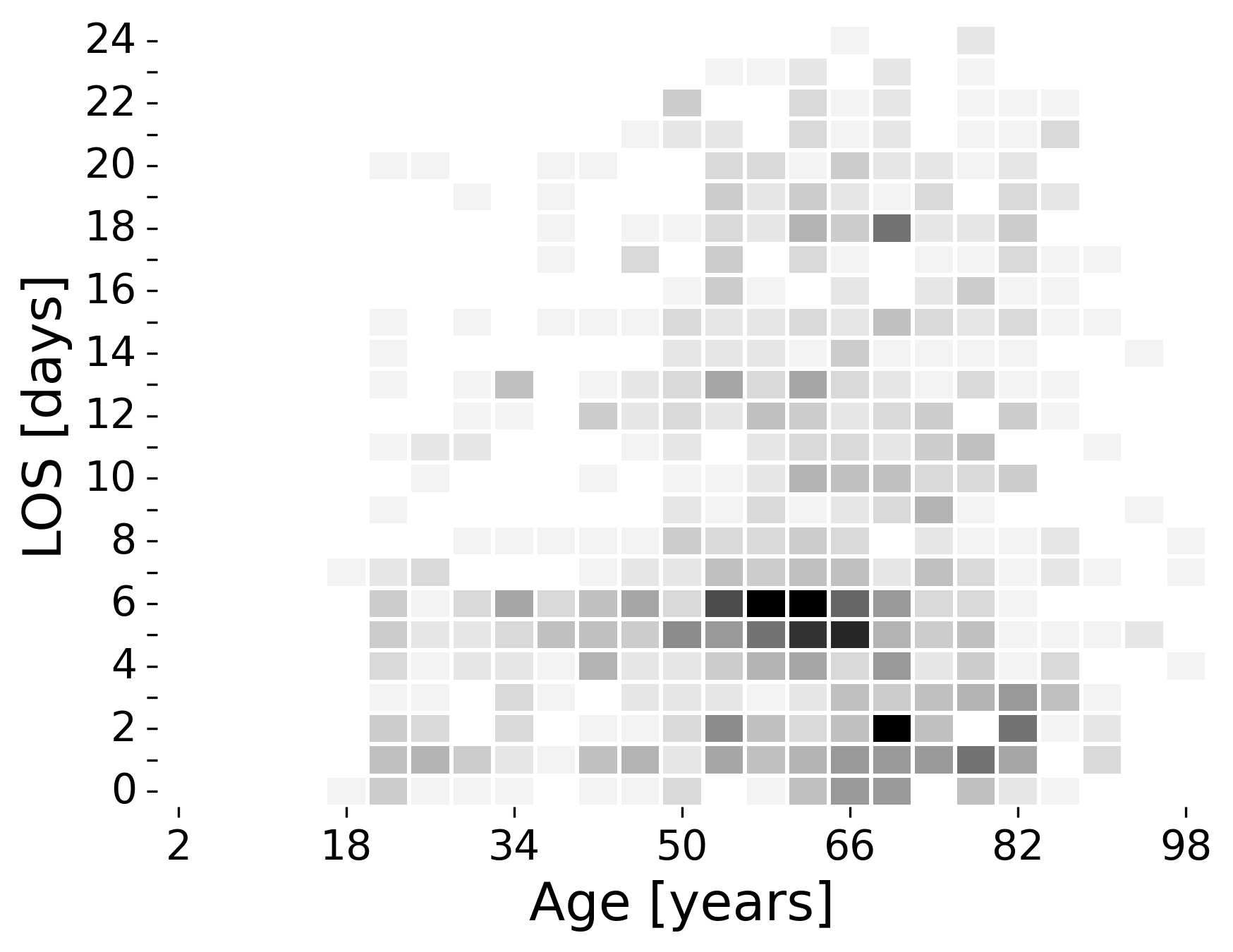}
        \caption{Type 5: outlier}
        \label{fig:rate:agelos5}
    \end{subfigure}
    \caption{Joint age and \ac{los} probability distribution where darker spots indicate higher probability}
    \label{fig:rate_params:agelos}
\end{figure}

For each ward, we compute the probability distribution of a patient's \ac{lor} and clustered the wards accordingly, yielding in five distinct distribution types (cf. \cref{fig:rate_params:lor}). The global \ac{lor} probability distribution is well approximated by a log-normal distribution with parameters $\mu=4.652$ and $\sigma=1.90$, estimated from the entire dataset\cref{note:lognormal_fit}.

\begin{figure}[ht!]
    \centering
       \begin{subfigure}[b]{0.49\textwidth}
        \includegraphics[scale=0.35]{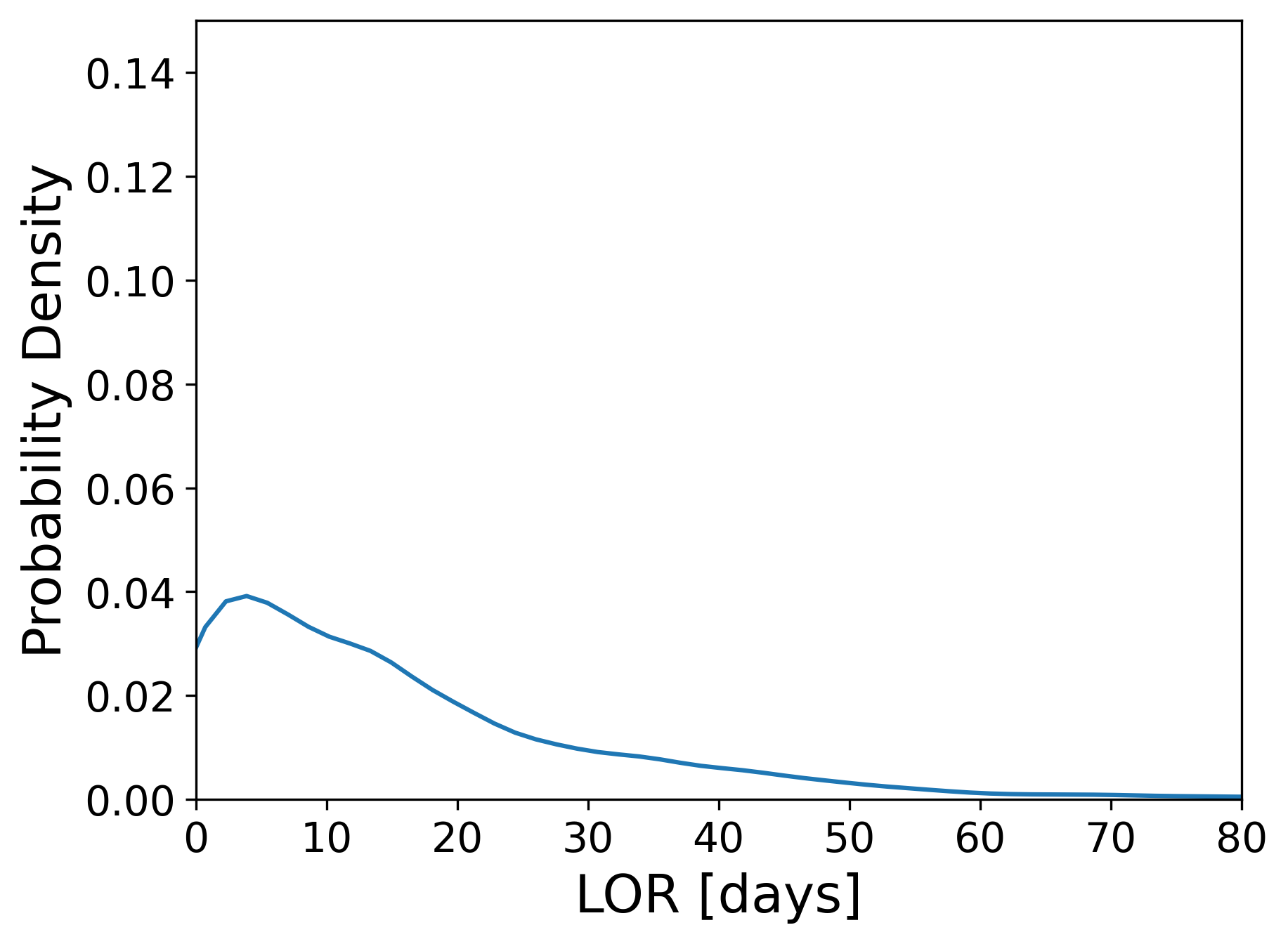}
        \caption{Type 1: evenly}
        \label{fig:rate:lor1}
    \end{subfigure}
    \begin{subfigure}[b]{0.49\textwidth}
        \includegraphics[scale=0.35]{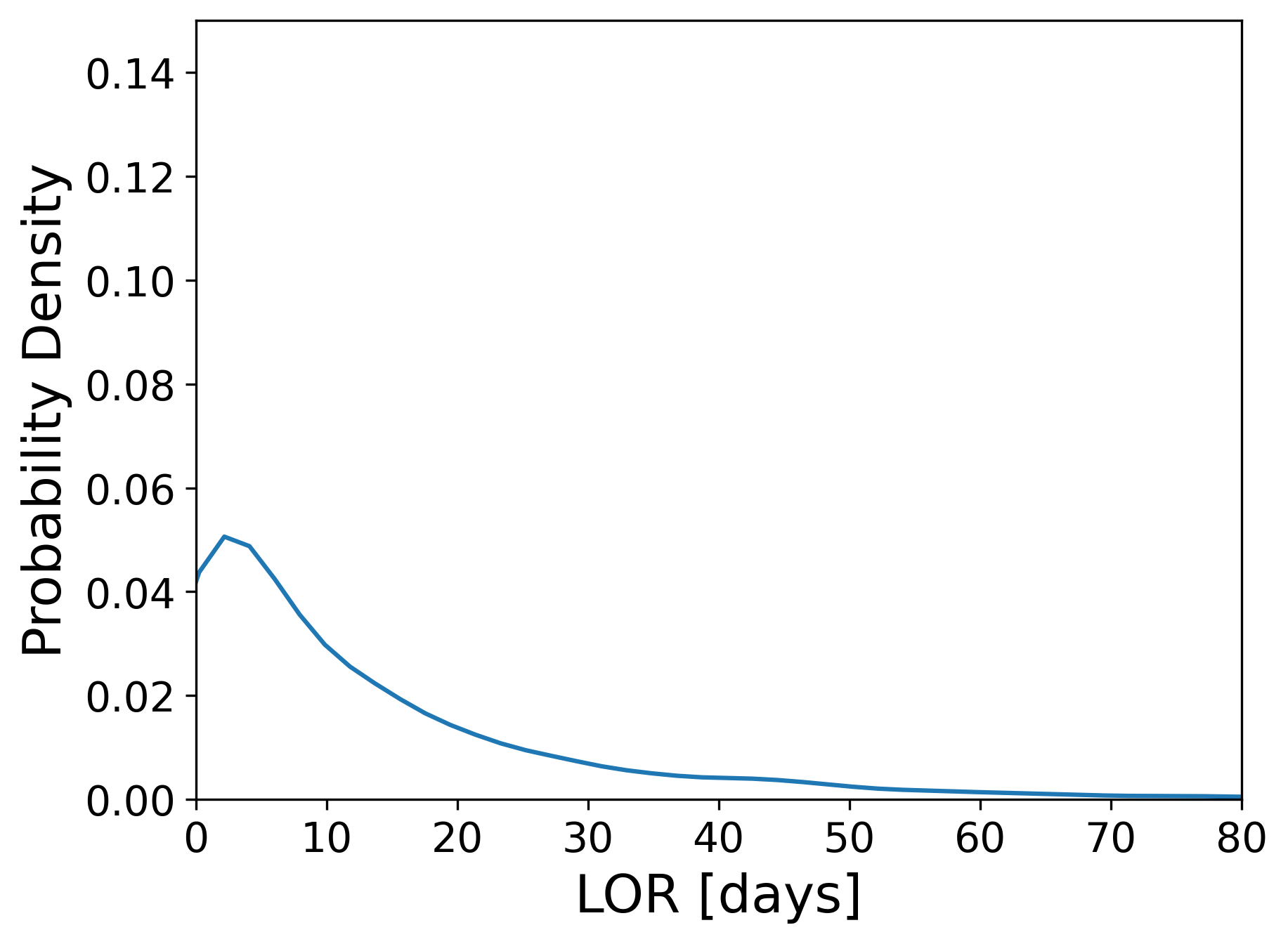}
        \caption{Type 2: middle peak}
        \label{fig:rate:lor2}
    \end{subfigure}
       \begin{subfigure}[b]{0.49\textwidth}
        \includegraphics[scale=0.35]{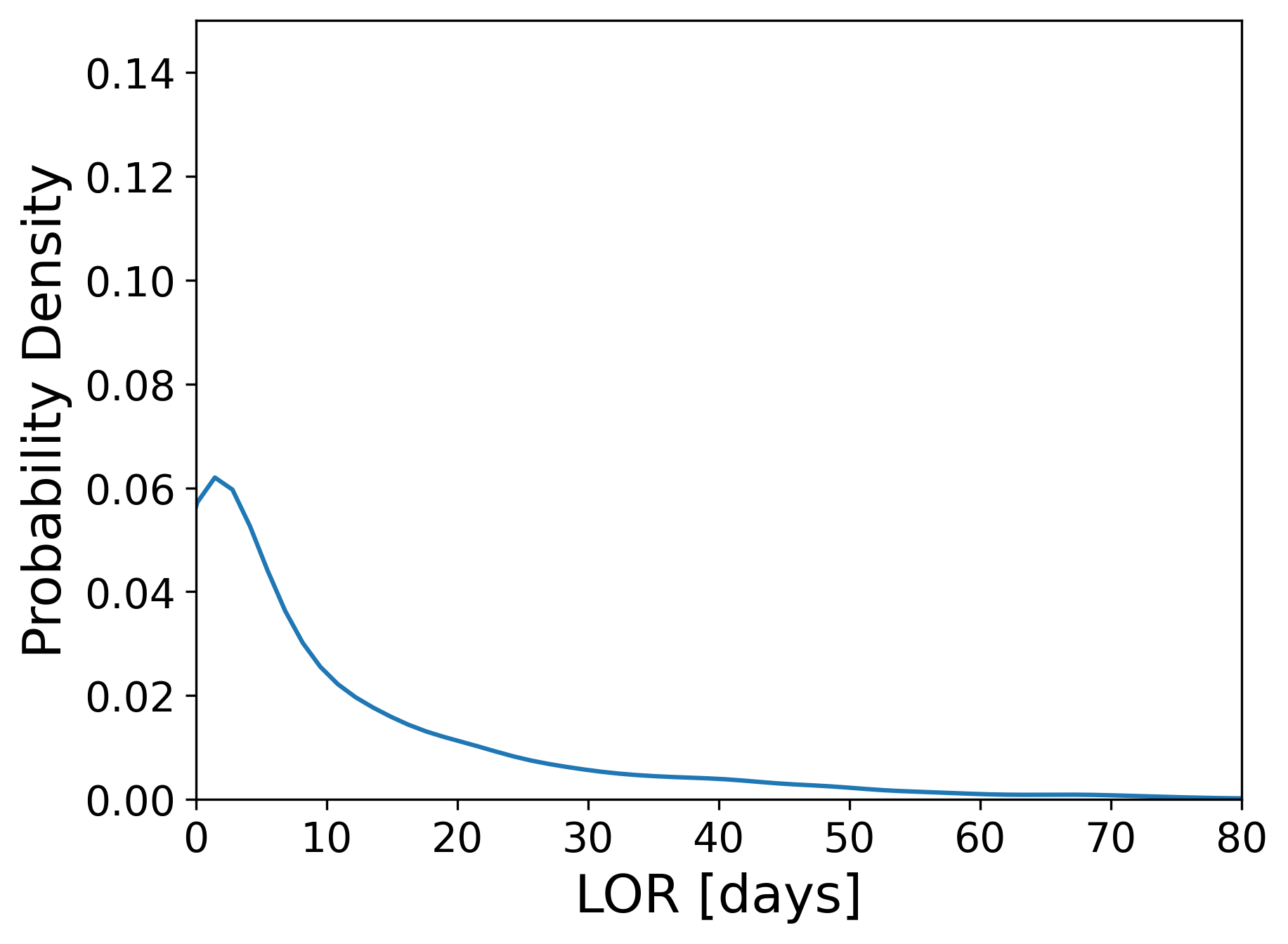}
        \caption{Type 3: higher peak}
        \label{fig:rate:lor3}
    \end{subfigure}
    \begin{subfigure}[b]{0.49\textwidth}
        \includegraphics[scale=0.35]{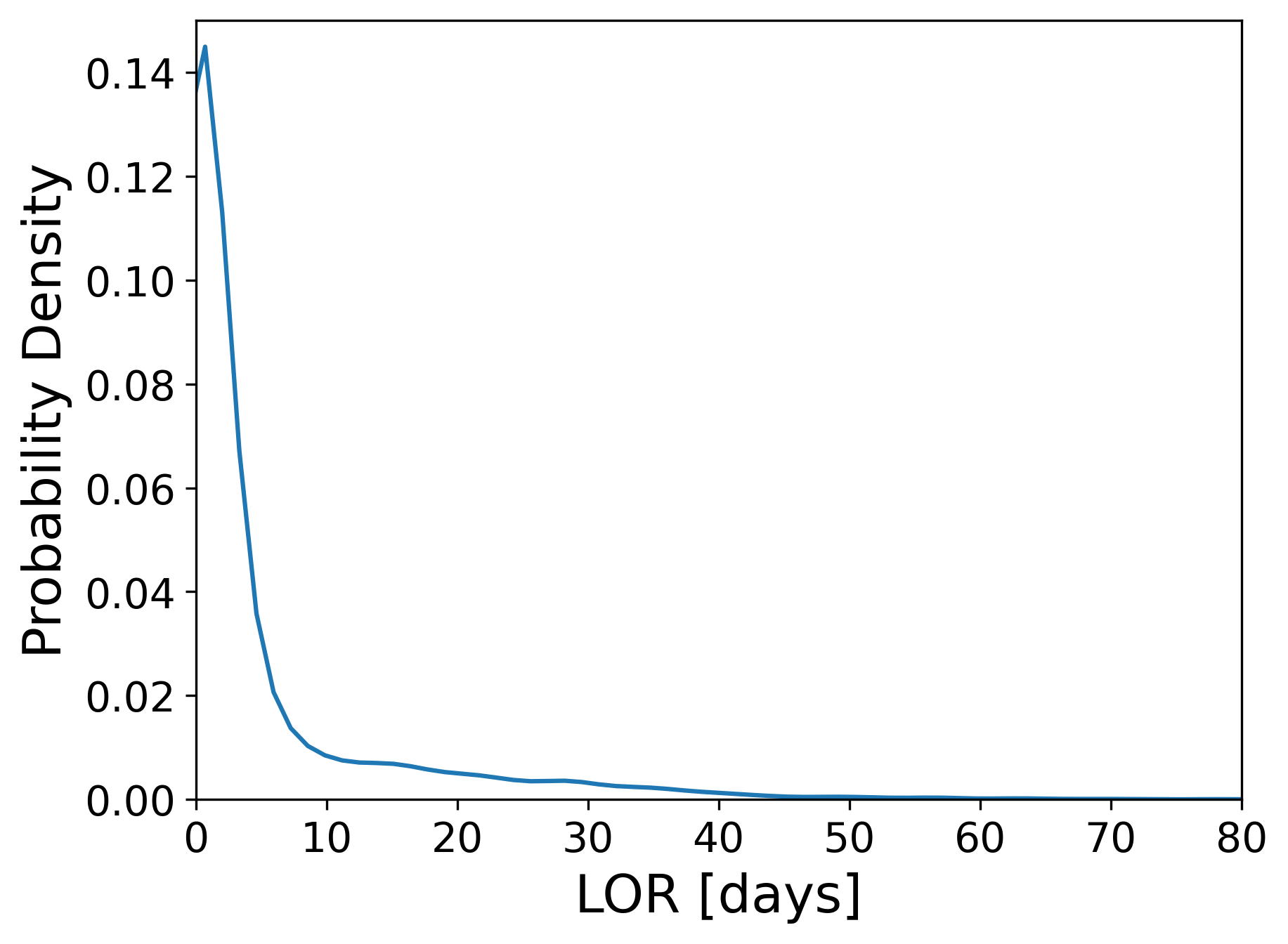}
        \caption{Type 4: very early peak}
        \label{fig:rate:lor4}
    \end{subfigure}
    \caption{Four types of \ac{lor} distributions}
    \label{fig:rate_params:lor}
\end{figure}
\FloatBarrier

Although the probabilities of emergency status, gender, single-room entitlement, and accompanying person also depend on the patient's age, we did not observe significant differences across different wards.
Therefore, we estimate for those only the general age-dependent probability based on the entire dataset.
Each distribution is well approximated by a cubic polynomial, and we used the \textit{Polyfit} method\footnote{\url{https://numpy.org/doc/stable/reference/generated/numpy.polynomial.polynomial.polyfit.html}} to determine the corresponding coefficients, cf. \cref{fig:rate_params:all}.\\
The rate of female patients is modeled as
\[2.58 \cdot 10^{-6}x^3 -3.17 \cdot 10^{-4}x^2 +8.95\cdot10^{-3}x +0.438.\]
The rate of emergency patients is modeled as
\[2.22 \cdot 10^{-6}x^3 +2.99\cdot 10^{-4}x^2 +1.02 \cdot 10^{-2}x +0.28.\]
The rate of single-room entitlement patients is modeled as
\[1.62 \cdot 10^{-6}x^3 +2.87 \cdot 10^{-4}x^2 +1.35 \cdot 10^{-2}x +0.27.\]
The rate of patients with an accompanying person is modeled as
\[5.65\cdot 10^{-8}x^3 +2.83\cdot 10^{-5}x^2 +3.02 \cdot 10^{-3}x +0.0978.\]

\begin{figure}[ht!]
    \centering
    \includegraphics[width=0.9\textwidth,keepaspectratio]{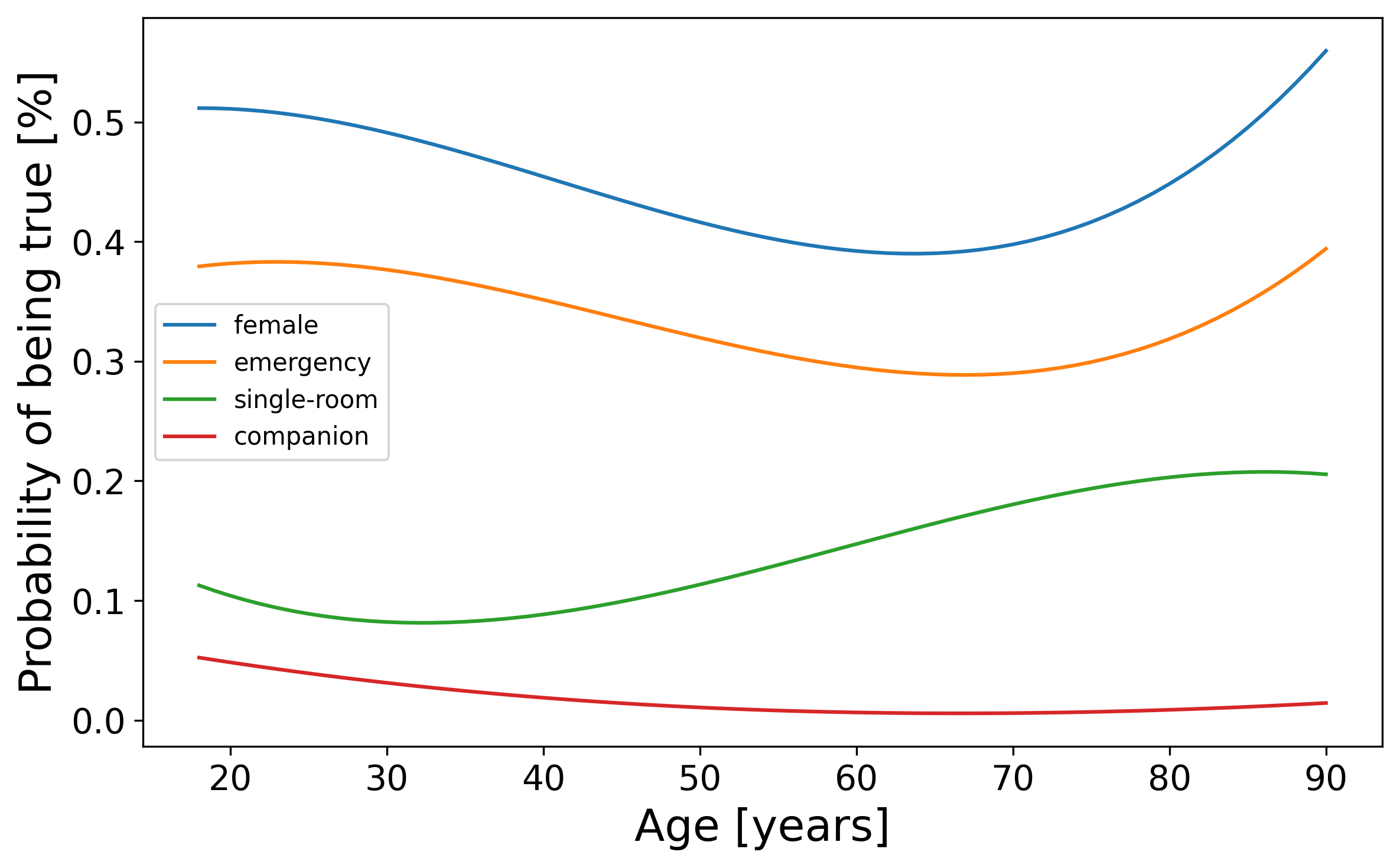}
    \caption{Polynomial functions to obtain a rate for each boolean attribute}
    \label{fig:rate_params:all}
\end{figure}

\subsection{Operating the instance generator}
The instance generator can be used either via the Python API or through the \ac{gui} which is accessible in form of a web application.
In this paper, we focus on the \ac{gui}, as it highlights all parameter choices.
An example demonstrating the Python API is included in the code repository (see \texttt{generator\textunderscore api.py} file).
The \ac{gui} guides users step by step through the parameter selection process as visualized in the flow chart \cref{fig:gui}.
In the first step, users may either load a previously saved parameter setting (referred to as a \emph{template}) or create a new one.
Loading a template sets all parameters to the stored values, but modifications are still possible.
In this step, users also define the planning horizon $\mathcal{T}$ and the desired load factor $\loadf_{\text{input}}$.
Moreover, they can choose between using the joint age-\ac{los} distributions or modeling  both attributes independently.
An additional option ensures feasibility with respect to gender-separated rooms, which is relevant in countries where mixed-gender rooms are not permitted.
If feasibility is not enforced, instances with a high load and long planning horizon are likely to be infeasible with respect to gender separation (see \cref{fig:boxplot_load_factor_planning_horizon}). 
Disabling the constraint also allows load factors exceeding $1$, i.e., $\loadf_{\text{input}} > 1$.

In the second step, users specify the number of rooms and their bed capacities. All choices are allowed since feasibility is ensured through a dynamic program, that will be introduced in \cref{sec:feas}, although excessively large room sizes can substantially increase computational time.

In the third step, users either select one of the five joint age-\ac{los} distributions (see \cref{fig:rate_params:agelos}), or if independent modeling was chosen in Step 1, select separate distributions for age and \ac{los}.
Options include custom uniform distributions, custom normal (for age) or log-normal distributions (for \ac{los}), or one of the ward-specific empirical distributions extracted from the data (see \cref{fig:rate_params:age,fig:rate_params:los}).
Default parameters correspond to the distribution fitted to the entire dataset.
Additionally, users may define custom minimum and maximum ages and \ac{los} values.

In the fourth step, users define a probability distribution for the \ac{lor}.
Available options include a custom uniform distribution, a custom log-normal distribution, or a ward-specific distribution (see \cref{fig:rate_params:lor}).
The default parameters of the log-normal distribution correspond to those fitted to the entire dataset.
Possible \ac{lor} values can be clipped but must be at least one, as patients with \ac{lor} of zero are defined as emergency patients; their frequency is specified in the next step.

In the fifth step, users specify the rates of female patients, emergency patients, patients entitled to a single-room, or patients accompanied by another person.
All rates depend on age and can be specified either per age class or directly via a polynomial function.
Rates defined per age class are internally converted into a cubic polynomial.
Default values correspond to those derived from the entire dataset (see \cref{fig:rate_params:all}).

In the last step, users specify how many instances to generate and the destination for storing them.

\begin{figure}
\small
\centering
\begin{tikzpicture}[node distance=2cm,every label/.style={align=left}]
\tikzstyle{startstop} = [rectangle, rounded corners, minimum width=2.5cm, minimum height=1cm, draw=black,text width=4cm]
\tikzstyle{process} = [rectangle, minimum width=1.5cm, minimum height=1cm, draw=black,text width=4cm]
\tikzstyle{decision} = [diamond, minimum width=2cm, minimum height=2cm, text centered, draw=black, text width = 2cm]
\tikzstyle{YN} = [minimum width=0cm, minimum height=0cm, text centered]
\tikzstyle{arrow} = [thick,->,>=stealth]

\node (start) [process] {\underline{Start}\\ - load/save template\\ - set planning horizon\\ - set load factor\\ - ensure feasibility w.r.t.\\ \phantom{-} gender-separation?\\ - use joint \ac{los}-age\\ \phantom{-} distribution?};
\node (beds) [process, right of = start,xshift=3cm] {\underline{Rooms}\\ - choose room capacities\\ \phantom{-} and quantities};
\draw [arrow] (start) -- (beds);
\node (indLOS) [decision, below of = beds, yshift=-1.5cm] {Joint \ac{los}-age distribution?};
\node (los) [process, left of = indLOS,xshift=-3cm] {\underline{\ac{los}}\\ - choose \ac{los} distribution};
\draw [arrow] (beds) -- (indLOS);
\draw [arrow] (indLOS) -- (los);
\node (no1) [YN,below of=indLOS, yshift=-0.6cm,xshift=0.4cm] {Yes};
\node (yes1) [YN,left of=indLOS, yshift=0.5cm,xshift=-0.4cm] {No};
\node (age) [process, below of = los,] {\underline{age}\\ - choose age distribution};
\draw [arrow] (los) -- (age);
\node (losage) [process, below of = indLOS,yshift=-2cm] {\underline{age \& \ac{los}}\\ - choose joint \ac{los}-age\\ \phantom{-} distribution};
\draw [arrow] (indLOS) -- (losage);
\node (lor) [process, below of = age,] {\underline{\ac{lor}}\\ - choose \ac{lor} distribution};
\draw [arrow] (losage) -- (lor);
\draw [arrow] (age) -- (lor);
\node (param) [process, below of = lor, yshift=-1cm] {\underline{other parameters}\\ - choose age-dependent\\ \phantom{-} distribution for\\ \; - gender\\ \; - emergency patients\\ \; - accompanying person\\ \; - single-room entitlement};
\draw [arrow] (lor) -- (param);
\node (gen) [process, below of = losage,yshift=-1cm] {\underline{Generate}\\ - set number of instances\\ \phantom{-} to generate\\ - set path to folder to\\ \phantom{-} store instances};
\draw [arrow] (param) -- (gen);
\end{tikzpicture}
\caption{Flow chart for the instance generator's \acs{gui}}
\label{fig:gui}
\end{figure}

\subsection{Generation process}
The output of the generator is a set of instances, each consisting of a set of patients with associated attributes, such that the prescribed load factor $\loadf_{\text{input}}$ is achieved for the given set of rooms.
Each instance is stored in a separate \texttt{json}-file using a format consistent with that described in \cite{Brandt2025}.

In the first step of the generation process, we create a large set of patients $\Ppool$ with their personal characteristics, particularly the \ac{lor} and \ac{los} values.
In the second step, we randomly select patients $p \in \Ppool$ and assign them an admission date $\arr(p)$ and add them to $\P$ until $\loadf_{\text{input}}$ is achieved.
Consequently, not all patients generated for the initial patient pool will be part of the final instance.
As previously stated, the registration and discharge dates, $\reg(p)$ and $\dis(p)$, can now be computed directly from the \ac{los} and \ac{lor} values of patient $p$.
We ensure $0 \le \reg(p) \le \arr(p) \le T$, but allow for $\dis(p)$ to exceed $T$ in order to preserve the original \ac{los} value.
Consequently, \cref{eq:los_to_patientCount} must be slightly adjusted to
\begin{align*}
    \sum_{p \in \P} \los(p) = \sum_{t \in \mathcal{T}} |\P(t)| + \Delta(\P) \quad
\end{align*}
with $ \Delta(\P) = \sum_{p \in \P} \max(0,\dis(p)-T)$.
We can upper bound $\Delta(\P)$ by
\begin{align*}
    \Delta(\P) = \sum_{p \in \P} \max(0,d(p)-T) = \sum_{p \in \P(T)} d(p)-T \leq \sum_{p \in \P(T)} \los(p).
\end{align*}
Our goal is to generate a set $\P \subseteq \Ppool$ that achieves the prescribed load factor $\loadf_{\text{input}}$.
Because $\P$ is not known in advance, we treat $|\P|$ as a random variable,
and thus $\loadf$ becomes a random variable as well.



To determine the required size of the patient pool $|\Ppool|$, we use the average \acl{los} $\Ex[\los]$ corresponding to the chosen \ac{los} distribution. We generate
\begin{align*}
    \left| \Ppool \right| = 2 \cdot \ceil*{\frac{\ctotal \cdot \nT \cdot \Ex[\loadf]}{\Ex[\los]}}
\end{align*}
patients in the initial patient pool $\Ppool$. We see that
\begin{align*}
     \ctotal \cdot \nT \cdot \loadf
    = \ctotal \cdot \sum_{t \in \mathcal{T}} \loadf_t 
    = \sum_{t \in \mathcal{T}} |\P(t)|
    = \sum_{p \in \P} \los(p) - \Delta(\P) 
    \geq \sum_{p \in \P \setminus \P(T)} \los(p).
\end{align*}
Using the linearity and monotonicity of the expected value, this guarantees that the pool is sufficiently large, i.e., $\left| \Ppool \right| \geq |\P|$, since
\begin{align*}
    \ctotal \cdot \nT \cdot \Ex[\loadf]
    \geq \Ex[\los] (|\P| - |\P(T)|)
    \geq \Ex[\los] \cdot \frac{1}{2}|\P|
\end{align*}
holds.
Note that the last step uses the condition $|\P(T)| \leq |\P|/2$ which becomes true for sufficiently long planning horizons $T$.

To construct the patient pool $\Ppool$, we begin by generating the age, \ac{los}, and \ac{lor} attributes.
When a predefined distribution type for an attribute is selected, the generator uses inverse transform sampling as follows:
Independent distribution types are represented by an empirical probability mass function defined by relative frequencies $f_i$ for each observed attribute value $i$ in the corresponding dataset.
From this, an empirical cumulative distribution function $F$ is constructed. 
The generator then draws a uniform random number $u \in [0,1]$, and selects the smallest value $i$ with $F(i) \geq u$.
For the dependent age–\ac{los} distributions, sampling is performed using a marginal–conditional decomposition.
First, an age class is selected using inverse transform sampling from the empirical marginal age distribution.
Each age class is associated with a pair of log-normal parameters $(\mu,\sigma)$ that characterize the conditional distribution of \ac{los}.
Subsequently, an age value is sampled uniformly from the corresponding age interval of the age class, and the \ac{los} is drawn from the associated log-normal distribution.
As established in the data analysis, age values are aggregated into classes of size at least 60 to guarantee sufficiently large sample sizes for the reliable estimation of the conditional \ac{los} distributions.

%

The remaining patient attributes (gender, single-room entitlement, presence of an accompanying person, and emergency status) are generated conditional on the sampled age.
For patients classified as emergencies, the \ac{lor} value is set to $0$.



Once the patient pool $\Ppool$ has been generated, we iterate over the planning horizon $\T$ to construct the final patient set $\P$.
For each day $t\in \T$, patients are selected from $\Ppool$, their admission date is set to $t$, and they are added to $\P$ as long as the resulting instance remains feasible and the target load has not been reached.
Once the admission date $\arr(p)$ of a patient is fixed, the corresponding registration date $\reg(p)$ and discharge date $\dis(p)$ are determined.
The generation process allows for day-to-day variability in the number of female and emergency patients. However, when averaged over the entire planning horizon, the prescribed proportions of female and emergency patients are approximately met.

In case, feasibility should be enforced, we verify the feasibility using a dynamic programming procedure that we will be explained in more detail in \cref{sec:feas}.
As a consequence, it may occur that the desired daily load cannot be achieved on certain days due to feasibility constraints.
To compensate for such deviations on subsequent days, we consider the cumulative load factor up to day $t \in \T$, defined as
\begin{align*}
    \loadf_{\leq t} = \frac{1}{t}\sum_{t^{\prime}=1}^t \loadf_{t^{\prime}} = \frac{1}{t \cdot \ctotal} \sum_{t^{\prime}=1}^t |\P(t^{\prime})|,
\end{align*}
and use it jointly with the current daily load $\loadf_t$ as a stopping criterion for admitting additional patients on day $t$.

A sketch of the generation procedure is provided in \cref{alg:patients:select_from_pool}.



\begin{algorithm}[ht!]
    \caption{Select from patient pool} \label{alg:patients:select_from_pool}
    \SetKwInput{KwInput}{Input} 
    \KwInput{Patient pool $\Ppool$, Planning horizon $\T$, load factor $\loadf$}
    \SetKwInput{KwOutput}{Output} 
    \KwOutput{Final patient set $\P \subseteq \P_{Pool}$}
    \medskip
    Initialize $\Ppool = \{p_1, \dots, p_m\}$, $\P \gets \emptyset$, $i \gets 1$\;
    \For{$t = 1, \dots, T$} {
        Set the admission date $\arr(p_i) \gets t$\;
        Calculate registration date $\reg(p_i)$, and discharge date $\dis(p_i)$\;
        \While{$\P \cup \{p_i\}$ is feasible \textbf{and} $\min(\loadf_t,\loadf_{\leq t}) \leq \loadf$ \textbf{and} $\loadf_{\leq T} \leq \loadf$} {
            $\P \gets \P \cup \{p_i\}$ \;
            \lIf{$i<m$}{
                Set $i \gets i+1$
            } \lElse {
                \Return $\P$
            }
            Set the admission date $\arr(p_i) \gets t$\;
            Calculate registration date $\reg(p_i)$, and discharge date $\dis(p_i)$\;
        }
    }
    \SetKwInput{KwReturn}{Return} 
    \KwReturn{$\P$}
\end{algorithm}

\section{Feasibility problem}\label{sec:feas}
In this section, we describe the mathematical foundation for guaranteeing feasibility of the generated \ac{pra} instances. 
We introduce the underlying feasibility problem, analyze its complexity for arbitrary room sizes, and review known polynomial cases. In addition, we present two new polynomially solvable cases, thereby extending the existing literature. 
Finally, we present a computational study that illustrates the practical necessity of enforcing feasibility, showing that without such guarantees instances with large load factors and long planning horizons are likely to be infeasible. 
While this section is not essential for understanding the instance generator, it provides valuable insights into how feasibility of the generated instances is ensured.\\

The feasibility can be addressed independently for each time period, since patient-to-room assignments from different periods can be combined if patients are allowed to change rooms arbitrarily often.
Therefore, we consider an arbitrary but fixed time period $t \in \T$ and denote the number of female patients in the hospital during this period as $\nf$ and the number of male patients as $\nm$. We also define $R_c := \vert \{ r \in \R \mid \croom = c\} \vert$ as the number of rooms with a specific capacity $c \in \N$. 
Brandt et al.\ define the feasibility problem of \ac{pra} for an arbitrary but fixed time period $t \in \T$ as follows \cite{Brandt_2025}:

\begin{definition}[Feasibility \ac{pra} Problem]
    Given the number of female and male patients $\nf, \nm \in \mathbb{N}_0$, and room capacities $\croom \in \mathbb{N}$ for $r \in \R$, does there exist a subset $S \subseteq \R$ of rooms such that it can host all female patients while the male patients fit into the remaining rooms, i.e.,
    \begin{align}
        \sum_{r \in S} \croom \geq \nf \text{ and } \sum_{r \in \R \setminus S} \croom \geq \nm?\label{Def:FeasibilityProblem}
    \end{align}
\end{definition}

Brandt et al.\ proved $\mathcal{NP}$-completeness for the feasibility problem with arbitrary room capacities \cite{brandt_2024}. 
However, we can check feasibility in pseudopolynomial time by transforming an instance of \ac{pra} into an equivalent instance of \ac{bkp}. 
In \ac{bkp}, we are given $n$ item types with a weight $w_i \in \N$, a number $u_i$ for the multiplicity of item $i$, as well as a knapsack capacity $W \in \N$. \ac{bkp} asks if there exists a multiset $S \subseteq \{1, \dots,n \}$ such that the capacity constraint $\sum_{i \in S} w_i \leq W$ is fulfilled, there are no more than $u_i$ items of type $i$ in $S$, and $\sum_{i \in S} w_i$ is maximized \cite{Garey_1990}.
A \ac{pra} instance can be modeled as a \ac{bkp} instance in which each existing room capacity $\croom$ defines an item type with weight equal to $\croom$ and multiplicity equal to $R_{\croom}$.
The \ac{pra} instance is feasible if and only if an (optimal) solution $S$ of the \ac{bkp} instance exists that fulfills the capacity constraint $\sum_{r \in S} \croom \leq \ctotal - \nm$ with maximum profit $\sum_{r \in S} \croom \geq \nf$. Consequently, we can use the dynamic programming algorithm for \ac{bkp} to prove feasibility for \ac{pra} in pseudopolynomial time, cf. \cite{Martello1990}\\

For large instances, it might be beneficial to check whether feasibility can be verified in polynomial time by using one of the following criteria instead of using the pseudopolynomial algorithm.
So far, Brandt et al.\ derived for constant room capacities $c \in \N$ a necessary and sufficient condition for feasibility \cite{brandt_2024}: 
\begin{align} \label{eq:constant_room_size}
    \ceil*{\frac{\nf}{c}} + \ceil*{\frac{\nm}{c}} \leq \vert \R \vert.
\end{align} 
Later, Brandt et al.\ presented two more polynomial cases \cite{Brandt_2025}: If hospitals have, in addition to rooms with constant capacity $c \in \N$, at least $c-1$ single rooms, a feasible assignment exists if and only if the number of patients does not exceed the hospital's total capacity, i.e.,
\begin{align} \label{eq:two_room_sizes}
    \nf + \nm \leq \ctotal.
\end{align}
In case of even room capacity $ \{2, 2c\}$ with $c \in \N_{\geq 2}$ and $R_2 \geq c-1$, the instance is feasible if and only if either (i) $\nf$ and $\nm$ are both even and $\nf + \nm \leq \ctotal$ or (ii)  $\nf + \nm < \ctotal$.

Additional polynomial cases can be derived from \ac{ssp}, \ac{bkp}, and Frobenius Problem. 
Similarly to \ac{ssp}, \ac{pra} asks for a subset $S \subseteq \R$ of the rooms whose total capacity lies in $\{ \nf, \dots, \ctotal - \nm \}$.
\ac{ssp} is polynomially solvable, for example, if the items
$\{a_1, \dots, a_n\}$ form a superincreasing sequence, i.e. if $\sum_{i = 1}^j a_i \leq a_{j+1}$ for $j = 1, \dots, n-1$,
or if the items form an arithmetic sequence, i.e., if $a_i = a_1 + (i-1)j$ for $i =2, \dots, n$ and $j \in \N$, cf. \cite{alfonsin_1998}.
Consequently, the feasibility problem of \ac{pra} is polynomially solvable if the room capacities form a superincreasing or arithmetic sequence.
Another polynomial case is if the room capacities form a chain since \ac{bkp} can be solved in polynomial time, if the weights $w_i$ form a chain, i.e. if $w_i$ divides $w_{i+1}$ for $i = 1, \dots, n-1$, cf. \cite{Verhaegh_1997}.

Lastly, \ac{pra} is related to the Frobenius problem, which is defined by a set of natural numbers $\{a_1, \dots, a_n\}\subseteq \N$ and $N \in \N$. The problem asks if coefficients $x_1,\dots, x_n \in \N_0$ exist with $\sum_{i = 1}^n x_i a_i = N$, cf. \cite{Beihoffer_2005}. Similar, an instance of \ac{pra} is feasible if and only if  there exist an $N \in \left\{ \nf, \dots, \ctotal - \nm \right\}$ and $ x_i \leq R_{a_i}$, for $i =1, \dots,n$, with $\sum_{i=1}^n x_ia_i = N$.
In case of arbitrary integers as coefficients in the Frobenius problem, there might be an infinite number of solutions to the equation. However, if we place additional constraints on the coefficients and the target value, we obtain uniqueness. For example, let $a_1, a_2 \in \N$ be coprime, i.e., the only positive integer that divides both of them is one. Then, for every integer $N \geq (a_1 -1)(a_2 -1)$, there exists exactly one pair $x_1,x_2 \in \N_0$ with $x_2 < a_1$ and $N= x_1 a_1 + x_2a_2$. The pair $x_1, x_2$ can be calculated in polynomial time, cf. \cite{Skupien_1993}.
Consequently, the feasibility problem with room capacities $\croom \in \{a_1, a_2\}$ for $a_1, a_2 \in \N$ coprime with $a_1 < a_2 $ , $R_{a_2} < a_1$, and $\nf \geq (a_1-1)(a_2-1)$ can be solved in polynomial time since we can calculate for each $N \in \left\{ \nf, \dots, \ctotal - \nm \right\}$ a unique representation. 
If there exists an $N$ such that its coefficients fulfill $x_1 \leq R_{a_1}$ and $x_2 \leq R_{a_2}$, the instance is feasible. Otherwise, the instance is infeasible because of the uniqueness of the representation.\\

Next, we discuss two new polynomial cases for the feasibility problem.
We first derive a feasibility criteria for instances with room sizes in $\{1, \dots, n \}$ and $\{b^0, \dots, b^n\}$ with $b \in \N_{\geq 2}$. This result is a special case of our insights from related problems since $1, \dots, n $ is an arithmetic sequence and $b^0, \dots, b^n$ a superincreasing sequence. However, this extension allows multiple rooms of the same size.
Afterwards, we generalize the result by multiplying the room sizes $\{1, \dots, n \}$ and $\{b^0, \dots, b^n\}$ with a scalar $a \in \N$.

\begin{lemma}\label{Lemma28}
    Let $(\nf, \nm, \R, c)$ be an instance with
    \begin{enumerate}[(i)]
        \item $\croom \in \{1, \dots , n\}$ for all $r \in \R$ and $R_i > 0$ for $i =1, \dots,n$ or
        \item $\croom \in \{ b^0, b^1, \dots, b^n \}$ for all $r \in \R$ and $R_{b^i} \geq b-1$ for $i = 0, \dots, n$
    \end{enumerate}
    and $b,n \in \mathbb{N}$ with $b \geq 2$. Then the instance is feasible if and only if the number of patients does not exceed the total capacity, i.e. if and only if 
    \begin{align}
        \nf +  \nm \leq \ctotal. \label{eq:WardCapacityConstraint}
    \end{align}
\end{lemma}

\begin{proof}
    If \cref{eq:WardCapacityConstraint} is violated, the instance is infeasible because at least one patient cannot be assigned to a room without violating the capacity constraint. Therefore, we assume \cref{eq:WardCapacityConstraint} holds true and show that the instance is feasible by constructing a set $S \subseteq \R$ that satisfies \cref{Def:FeasibilityProblem}.
    We construct $S$ such that $\sum_{r \in S} \croom = \nf$ holds. \cref{eq:WardCapacityConstraint} implies $\sum_{r \in \R \setminus S} \croom \geq \nm$, thus making the instance feasible.

    We will show the existence of $S$ by induction over $\nf \leq \ctotal$.
    For the base case $\nf = 1$, we define $S = \{r_1\}$ for a single-room $r_1 \in \R$. Consider now the induction step.
    Let $S' \subseteq \R$ with $\sum_{r \in S'} \croom = \nf -1$. 
    To construct $S$ such that $\sum_{r \in S} \croom = F$, we choose a room $r' \in \R \setminus S'$ with minimum capacity $c_{r'}$.
    If $c_{r'} = 1$, then set $S = S' \cup \{r' \}.$ Else $c_{r'}>1$ and $S'$ contains a room of each size smaller than $c_{r'}$. 
    In case (i), it follows that $S'$ contains a room $r''$ with capacity $c_{r''} = c_{r'} -1$. Consequently, $S$ can be defined as $S'\setminus \{ r'' \} \cup \{r' \}$, cf. \cref{fig:Lemma28_Case1}.
    \begin{figure}[h]
	\centering
	\begin{minipage}[t]{.39\linewidth}
		\centering
		\begin{tikzpicture}[scale=.6]	
		
		\node[blacksquare, minimum width=.5cm, minimum height=.5cm, fill=gruen] at (0,2) {};
		\node[grayball] at (0,2) {};
		
		\node[blacksquare, minimum width=1cm, minimum height=.5cm, fill=gruen] at (1.55,2) {};
		\node[grayball] at (1.15,2) {};
		\node[grayball] at (1.95,2) {};
		
		\node[blacksquare, minimum width=1.5cm, minimum height=.5cm] at (3.9,2) {};
		\foreach \x in {1,2,0} {
			\node[grayball] at (\x*.8+3.1,2) {};
		}
		
        \node[label={below:\textcolor{black}{\footnotesize$r'$}}] at (4,1.8) {};
	
			\node[label={below:\textcolor{black}{\footnotesize$r''$}}] at (1.6,1.8) {};
			
			\draw[edge, <->] (1.7,2.6) to[bend left] (3.8,2.6) {};
		
		\end{tikzpicture}
	\end{minipage}
	\begin{minipage}[t]{.1 \linewidth}
		\centering
		\begin{tikzpicture}[scale=.6]
			\draw[edge, ->, very thick] (-0.9,1) --(0.9,1) {};

                \node[blacksquare, minimum width=.5cm, minimum height=.5cm, draw=none] at (0,2) {};
			\node[blacksquare, minimum width=.5cm, minimum height=.5cm, draw=none] at (0,0.3) {};
			
		\end{tikzpicture}
	\end{minipage}
	\begin{minipage}[t]{.39\linewidth}
		\centering			
		\begin{tikzpicture}[scale=.6]	
		
		\node[blacksquare, minimum width=.5cm, minimum height=.5cm, fill=gruen] at (0,2) {};
		\node[grayball] at (0,2) {};
		
		\node[blacksquare, minimum width=1cm, minimum height=.5cm] at (1.55,2) {};
		\node[grayball] at (1.15,2) {};
		\node[grayball] at (1.95,2) {};
		
		\node[blacksquare, minimum width=1.5cm, minimum height=.5cm, fill=gruen] at (3.9,2) {};
		\foreach \x in {1,2,0} {
			\node[grayball] at (\x*.8+3.1,2) {};
		}
		
		\node[label={below:\textcolor{black}{\footnotesize$r'$}}] at (4,1.8) {};

		\node[label={below:\textcolor{black}{\footnotesize$r''$}}] at (1.6,1.8) {};
		
		\draw[edge, <->] (1.7,2.6) to[bend left] (3.8,2.6) {};

		\end{tikzpicture}
	\end{minipage}

	\begin{minipage}[t]{.39\linewidth}
		\centering
		\begin{tikzpicture}[scale=.6]
			\node[blacksquare, minimum width=.45cm, minimum height=.45cm, fill=gruen, draw=none, label={right:\footnotesize Set $S'$ with capacity $\nf-1$}] at (1.5,-1) {};	
		\end{tikzpicture}
	\end{minipage}
	\begin{minipage}[t]{.1 \linewidth}
		\centering
		\begin{tikzpicture}[scale=.6]
		
		\end{tikzpicture}
	\end{minipage}
	\begin{minipage}[t]{.39\linewidth}
		\centering			
		\begin{tikzpicture}[scale=.6]
			\node at (1.5,-.2) {};
			\node[blacksquare, minimum width=.45cm, minimum height=.45cm, fill=gruen, draw=none, label={right:\footnotesize Set $S$ with capacity $\nf \textcolor{white}{-1}$}] at (1.5,-1) {};	
		\end{tikzpicture}
	\end{minipage}
	\caption{Induction step for case (i)}
        \label{fig:Lemma28_Case1}
\end{figure}
    In case (ii), the capacity of room $r'$ is given by $c_{r'} = b^\lambda$ for $\lambda \in \mathbb{N}$. Here, a subset $T \subseteq S'$ exists which includes exactly $b-1$ rooms of each capacity smaller than $c_{r'}$. The total capacity of $T$ is 
    \begin{align*}
    (b-1)\cdot \sum_{i = 0}^{\lambda-1} b^i = (b-1)\cdot\frac{1-b^\lambda}{1-b} = c_{r'} -1.    
    \end{align*}
    Thus, we define $S = S' \setminus T \cup \{r^{\prime}\}$, cf. \cref{fig:Lemma28_Case2}.
    \begin{figure}[h]
	\centering
	\begin{minipage}[t]{.39\linewidth}
		\centering
		\begin{tikzpicture}[scale=.6]	
		
		\node[blacksquare, minimum width=.5cm, minimum height=.5cm, fill=gruen] at (0,2) {};
		\node[grayball] at (0,2) {};
		
		\node[blacksquare, minimum width=1cm, minimum height=.5cm, fill=gruen] at (1.55,2) {};
		\node[grayball] at (1.15,2) {};
		\node[grayball] at (1.95,2) {};
		
		\node[blacksquare, minimum width=2cm, minimum height=.5cm] at (4.4,2) {};
		\foreach \x in {1,2,3,0} {
			\node[grayball] at (\x*.8+3.2,2) {};
		}

		\node[label={below:\textcolor{black}{\footnotesize$r'$}}] at (4.4,1.7) {};
		
		\node[blacksquare, minimum width=1.9cm, minimum height=.7cm, draw=red, rounded corners=.1cm] at (.99,2) {};
		\node[label={below:\textcolor{red}{\footnotesize$T$}}] at (1,1.7) {};
			
		\draw[edge, <->] (1.,2.8) to[bend left] (4.3,2.7) {};
		
		\node[blacksquare, minimum width=.45cm, minimum height=.45cm, fill=gruen, draw=none, label={right:\footnotesize Set $S'$ with capacity $\nf-1$}] at (0,-0.3) {};	
		
		\end{tikzpicture}
	\end{minipage}
	\begin{minipage}[t]{.1 \linewidth}
		\centering
		\begin{tikzpicture}[scale=.6]
			\draw[edge, ->, very thick] (-0.9,1.1) --(0.9,1.1) {};
			
			\node[blacksquare, minimum width=.5cm, minimum height=.5cm, draw=none] at (0,2) {};
			\node[blacksquare, minimum width=.5cm, minimum height=.5cm, draw=none] at (0,0.3) {};
			\node[blacksquare, minimum width=.4cm, minimum height=.4cm, draw=none]  at (0,-1.2) {};
		\end{tikzpicture}
	\end{minipage}
	\begin{minipage}[t]{.39\linewidth}
		\centering			
		\begin{tikzpicture}[scale=.6]	
		
		\node[blacksquare, minimum width=.5cm, minimum height=.5cm] at (0,2) {};
		\node[grayball] at (0,2) {};
		
		\node[blacksquare, minimum width=1cm, minimum height=.5cm] at (1.55,2) {};
		\node[grayball] at (1.15,2) {};
		\node[grayball] at (1.95,2) {};
		
		\node[blacksquare, minimum width=2cm, minimum height=.5cm, fill=gruen] at (4.4,2) {};
		\foreach \x in {1,2,3,0} {
			\node[grayball] at (\x*.8+3.2,2) {};
		}

		\node[label={below:\textcolor{black}{\footnotesize$r'$}}] at (4.4,1.7) {};

		\node[blacksquare, minimum width=1.9cm, minimum height=.7cm, draw=red, rounded corners=.1cm] at (.99,2) {};
		\node[label={below:\textcolor{red}{\footnotesize$T$}}] at (1,1.7) {};
		
		\draw[edge, <->] (1.,2.8) to[bend left] (4.3,2.7) {};
		
		\node[blacksquare, minimum width=.45cm, minimum height=.45cm, fill=gruen, draw=none, label={right:\footnotesize Set $S$ with capacity $\nf$}] at (0,-0.3) {};	
		\end{tikzpicture}
	\end{minipage}

			
			
	\caption{Induction step for case (ii) with $b=2$}
        \label{fig:Lemma28_Case2}
\end{figure}
    In both cases, we can construct a set $S \subseteq  \R$ with $\sum_{r \in S} \croom = \nf$, which concludes the induction. 
\end{proof}

As already mentioned earlier, we can use the above result to derive a feasibility criteria for the more general case in which room capacities are multiplied with a scalar $a \in \N$.

\begin{lemma}
    Let $(\nf, \nm, \R, c)$ be an instance with
    \begin{enumerate}[(i)]
        \item $\croom  \in \{a, na \}$ for all $r \in \R$ and $R_a \geq n-1$, or
        \item $\croom \in \{a, 2a, \dots, na\}$ for all $r \in \R$ and $R_{ia} >0$ for $i = 1, \dots, n$, or 
        \item $\croom \in \{b^0a, b^1a, \dots, b^na\}$ for all $r \in \R$ and $R_{b^ia} \geq b-1$ for $i = 0, \dots, n$
    \end{enumerate}
    and $a,b,n \in \mathbb{N}$ with $b \geq 2$. Then the instance is feasible if and only if
    \begin{align}
        \ceil*{\frac{\nf}{a}} + \ceil*{\frac{\nm}{a}} \leq \frac{1}{a} \ctotal. \label{eq:FeasibilityLemma29}
    \end{align}
\end{lemma}

\begin{proof}
    Let $\mathcal{I} =( \nf, \nm, \R, c)$ be an instance.
    We construct an auxiliary instance $\mathcal{I}^* = ( \nf^*, \nm^*, \R^*, c^*)$ with $\nf^* = \nf$ and $\nm^* = \nm.$
    For every room $r \in \R$, the set $\R^*$ contains $\frac{\croom}{a} $ rooms of size $a$.
    As all rooms in $\R^*$ have equal capacity $a$, according to \cref{eq:constant_room_size}, the instance $\mathcal{I}^*$ is feasible if and only if 
    \begin{align*}
        \ceil*{\frac{\nf^*}{a}} + \ceil*{\frac{\nm^*}{a}}  \leq \vert \R^* \vert
    \end{align*}
    which is equivalent to 
    \begin{align*}
        \ceil*{\frac{\nf}{a}} + \ceil*{\frac{\nm}{a}} \leq \frac{1}{a} \ctotal.
    \end{align*}

    Thus, it remains to show that $\mathcal{I}$ is feasible if and only if $\mathcal{I}^*$ is feasible.
    Let $S \subseteq \R$ be a feasible assignment for $\mathcal{I}$. Since all room in $\R^*$ have equal capacity, we define $S^* \subseteq \R^*$ with $\vert S^* \vert = \frac{1}{a} \left( \sum_{r \in S} \croom \right)$, cf. \cref{fig:cn2V2}.
    The set $S^*$ defines a feasible assignment for $\mathcal{I}^*$ as it holds
    \begin{align*}
        \sum_{r \in S^*} \croom &= \sum_{r \in S^*} a = a \cdot  \vert S^*\vert = \sum_{r \in S} \croom \geq \nf = \nf^*, \text{ and }\\
        \sum_{r \in \R^* \setminus S^*} \croom &= a \cdot ( \vert \R^* \vert - \vert S^* \vert)  = \sum_{r \in \R \setminus S} \croom \geq \nm = \nm^*.
    \end{align*}
    \begin{figure}[h]
\centering
	\begin{minipage}[t]{.4 \linewidth}
		\scalebox{.95}{
        \centering
		\begin{tikzpicture}[scale=.6]	
			
			\node[blacksquare, minimum width=.5cm, minimum height=1cm] at (0,2) {};
			\node[greenball] at (0,2.4) {};
			\node[greenball] at (0,1.6) {};
			
			\node[blacksquare, minimum width=1cm, minimum height=1cm] at (1.9,2) {};
			\node[redball] at (1.5,2.4) {};
			\node[redball] at (2.3,1.6) {};
			\node[redball] at (1.5,1.6) {};
			\node[redball] at (2.3,2.4) {};
			
			\node[blacksquare, minimum width=1cm, minimum height=1cm] at (4.3,2) {};
			\foreach \x in {1,0} {
				\node[redball] at (\x*.8+3.9,2.4) {};
				\node[redball] at (\x*.8+3.9,1.6) {};
			}
            \node[grayball] at (.8+3.9,1.6) {};
			
			\node[blacksquare, minimum width=1.5cm, minimum height=1cm] at (7.1,2) {};
			\foreach \x in {1,2,0} {
				\node[greenball] at (\x*.8+6.3,2.4) {};
				\node[greenball] at (\x*.8+6.3,1.6) {};
			}
			
			\node[blacksquare, minimum width=.5cm, minimum height=1cm, draw=none] at (6.2+2,2) {};

		\end{tikzpicture}}
		\subcaption{original instance $\mathcal{I}$}
	\end{minipage}
\hfill
	\begin{minipage}[t]{.05\linewidth}
    \centering
		\centering
		\begin{tikzpicture}[scale=.6]
		\draw[edge, -> ] (12,.8)--(13.5,.8) {};
		
		\node[blacksquare, minimum width=.5cm, minimum height=1cm, draw=none] at (13,.8) {};
		\end{tikzpicture}
	\end{minipage}
\hfill
    \begin{minipage}[t]{.45 \linewidth}
		\centering
		\scalebox{.95}{
		\begin{tikzpicture}[scale=.6]	
		
			\node[blacksquare, minimum width=.5cm, minimum height=1cm] at (0,2) {};
			\node[greenball] at (0,2.4) {};
			\node[greenball] at (0,1.6) {};

			\node[blacksquare, minimum width=.5cm, minimum height=1cm] at (1.4,2) {};
			\node[blacksquare, minimum width=.5cm, minimum height=1cm] at (2.4,2) {};
			
			\node[redball] at (1.4,2.4) {};
			\node[redball] at (2.4,1.6) {};
			\node[redball] at (1.4,1.6) {};
			\node[redball] at (2.4,2.4) {};
			
			\node[blacksquare, minimum width=.5cm, minimum height=1cm] at (4.8,2) {};
			\node[blacksquare, minimum width=.5cm, minimum height=1cm] at (3.8,2) {};
					
			\foreach \x in {1,0} {
				\node[redball] at (\x+3.8,2.4) {};
				\node[redball] at (\x+3.8,1.6) {};
			}
            \node[grayball] at (1+3.8,1.6) {};

			\foreach \x in {1,0,2} {
				\node[blacksquare, minimum width=.5cm, minimum height=1cm] at (6.2+\x,2) {};
				\node[greenball] at (\x+6.2,2.4) {};
				\node[greenball] at (\x+6.2,1.6) {};
			}

		\end{tikzpicture}}
		\subcaption{auxiliary instance $\mathcal{I}^*$ }
	\end{minipage}
	
	\caption{Construction of $S^*$ in case $\croom \in \{a, 2a, \dots, na\}$ for $r \in \R$}
    \label{fig:cn2V2}
\end{figure}	
	
    
    Let $S^* \subseteq \R^*$ be a feasible assignment for $\mathcal{I}^*$. 
    First, we assume w.l.o.g. that female and male patients are assigned to rooms as compact as possible, meaning that $\lceil \frac{\nf^*}{a} \rceil$ and $\lceil \frac{\nm^*}{a} \rceil$ rooms are assigned to female and male patients.
    In order to derive a feasible assignment for $\mathcal{I}$, we need to reverse the splitting of rooms, meaning we must combine rooms until the original room sizes in $\R$ are restored. However, to guarantee feasibility, only rooms containing patients of the same gender can be combined. This problem can be formulated as a new instance $\mathcal{I}' = (\nf', \nm', \R', c')$ of the feasibility problem:
    We define $\nf' = \ceil{\frac{\nf^*}{a}} = \ceil{\frac{\nf}{a}}$ and $\nm' = \ceil{\frac{\nm^*}{a}} = \ceil{\frac{\nm}{a}}$. For every room $r \in \R$, a room $r'$ of capacity $c_{r'}= \frac{\croom}{a}$ is added to $\R'$. We obtain $\vert \R' \vert = \vert \R \vert$ and $c_{r'} = \frac{\croom}{a} \in \N $ for all $ r \in \R.$
    A feasible assignment of $\mathcal{I}'$  corresponds to a feasible assignment of $\mathcal{I}$ and vice versa, cf. \cref{fig:cn3V2}:
    Given a feasible assignment $S' \subseteq \R'$, we define $S= S'$ and obtain 
    \begin{align*}
        \sum_{r \in S} \croom = a \sum_{r \in S'} \croom' \geq a \nf' \geq \nf \ \text{ and } \sum_{ r \in \R \setminus S} \croom = a\sum_{r \in \R' \setminus S'}  \croom' \geq a \nm' \geq \nm.
    \end{align*}
    Similarly, starting from a feasible assignment $\mathcal{I}$, we can construct a feasible assignment for $\mathcal{I}'$ since it holds
    \begin{align*}
        \sum_{r \in S'} \croom' = \sum_{r \in S} \frac{1}{a} \croom \geq \ceil*{\frac{\nf}{a}} = \nf' \ \text{ and } \sum_{r \in \R' \setminus S'} \croom' = \sum_{ r \in \R \setminus S} \frac{1}{a} \croom \geq \ceil*{\frac{\nm}{a}} = \nm',
    \end{align*}
    whereby we used the integrality of $\sum_{r \in S} \frac{1}{a} \croom$ and $\sum_{r \in \R \setminus S} \frac{1}{a} \croom$. \\

    \begin{figure}[h]
	\begin{minipage}[t]{.43 \linewidth}
		\centering
		\scalebox{.95}{
		\begin{tikzpicture}[scale=.6]	
		
			\node[blacksquare, minimum width=.5cm, minimum height=1cm] at (0,2) {};
			\node[greenball] at (0,2.4) {};
			\node[greenball] at (0,1.6) {};

			\node[blacksquare, minimum width=.5cm, minimum height=1cm] at (1.4,2) {};
			\node[blacksquare, minimum width=.5cm, minimum height=1cm] at (2.4,2) {};
			
            \foreach \x in {1,0} {
				\node[greenball] at (\x+1.4,2.4) {};
				\node[greenball] at (\x+1.4,1.6) {};
			}
			
			\node[blacksquare, minimum width=.5cm, minimum height=1cm] at (4.8,2) {};
			\node[blacksquare, minimum width=.5cm, minimum height=1cm] at (3.8,2) {};
					
            \node[greenball] at (3.8,2.4) {};
            \node[greenball] at (3.8,1.6) {};
            \node[redball] at (4.8,2.4) {};
            \node[redball] at (4.8,1.6) {};

			\foreach \x in {1,0,2} {
				\node[blacksquare, minimum width=.5cm, minimum height=1cm] at (6.2+\x,2) {};
				\node[redball] at (\x+6.2,2.4) {};
				\node[redball] at (\x+6.2,1.6) {};
			}
            \node[grayball] at (2+6.2,1.6) {};

		\end{tikzpicture}}
		\subcaption*{feasible assignment of $\mathcal{I}^*$}
	\end{minipage}
\hfill
	\begin{minipage}[t]{.05\linewidth}
		\centering
		\begin{tikzpicture}[scale=.6]
		\draw[edge, -> ] (12,.8)--(13.5,.8) {};
		
		\node[blacksquare, minimum width=.5cm, minimum height=1cm, draw=none] at (13,.8) {};
		\end{tikzpicture}
	\end{minipage}
\hfill
	\begin{minipage}[t]{.44 \linewidth}
		\scalebox{.95}{
		\begin{tikzpicture}[scale=.6]	
			
			\node[blacksquare, minimum width=.5cm, minimum height=1cm] at (0,2) {};
			\node[grayball] at (0,2.4) {};
			\node[grayball] at (0,1.6) {};
			
			\node[blacksquare, minimum width=1cm, minimum height=1cm] at (1.9,2) {};
			\node[grayball] at (1.5,2.4) {};
			\node[grayball] at (2.3,1.6) {};
			\node[grayball] at (1.5,1.6) {};
			\node[grayball] at (2.3,2.4) {};
			
			\node[blacksquare, minimum width=1cm, minimum height=1cm] at (4.3,2) {};
			\foreach \x in {1,0} {
				\node[grayball] at (\x*.8+3.9,2.4) {};
				\node[grayball] at (\x*.8+3.9,1.6) {};
			}
			
			\node[blacksquare, minimum width=1.5cm, minimum height=1cm] at (7.1,2) {};
			\foreach \x in {1,2,0} {
				\node[grayball] at (\x*.8+6.3,2.4) {};
				\node[grayball] at (\x*.8+6.3,1.6) {};
			}
			
			\node[blacksquare, minimum width=.5cm, minimum height=1cm, draw=none] at (6.2+2,2) {};

		\end{tikzpicture}}
		\subcaption*{rooms $\R$ }
	\end{minipage}

	\begin{minipage}[t]{.43 \linewidth}
		\centering
		\scalebox{.95}{
		\begin{tikzpicture}[scale=.6]	
		
			\node[greenball] at (0,1.6) {};
            
			\node[greenball] at (2.4,1.6) {};
			\node[greenball] at (1.4,1.6) {};
				
			\node[greenball] at (3.8,1.6) {};
            \node[redball] at (1+3.8,1.6) {};

			\foreach \x in {1,0,2} {
				\node[redball] at (\x+6.2,1.6) {};
			}

            \node[blacksquare, minimum width=.5cm, minimum height=.5cm, draw=none] at (5,1.6) {};

            \draw[edge, <->]  (4.1,2.5)-- (4.1,4) {};
		\end{tikzpicture}}
		\subcaption*{female and male patients $\nf'$ and $\nm'$}
	\end{minipage}
	\hfill
	\begin{minipage}[t]{.05\linewidth}
		\centering
		\begin{tikzpicture}[scale=.6]
        \node[blacksquare, minimum width=.5cm, minimum height=.5cm, draw=none] at (12.5,1.6) {};
		\draw[edge, -> ] (12,1.6)--(13.5,1.6) {};
		\end{tikzpicture}
	\end{minipage}
	\hfill
	\begin{minipage}[t]{.455 \linewidth}	
		\centering
		\scalebox{.95}{
		\begin{tikzpicture}[scale=.6]	
			
			\node[blacksquare, minimum width=.5cm, minimum height=.5cm] at (0,1.6) {};
			\node[grayball] at (0,1.6) {};
			
			\node[blacksquare, minimum width=1cm, minimum height=.5cm] at (1.9,1.6) {};
			\node[grayball] at (2.3,1.6) {};
			\node[grayball] at (1.5,1.6) {};
			
			\node[blacksquare, minimum width=1cm, minimum height=.5cm] at (4.3,1.6) {};
			\foreach \x in {1,0} {
				\node[grayball] at (\x*.8+3.9,1.6) {};
			}
			
			\node[blacksquare, minimum width=1.5cm, minimum height=.5cm] at (7.1,1.6) {};
			\foreach \x in {1,2,0} {
				\node[grayball] at (\x*.8+6.3,1.6) {};
			}

            \draw[edge, <->]  (3.6,2.5)-- (3.6,4) {};
		\end{tikzpicture}}
       \subcaption*{rooms $\R'$}
	\end{minipage}
	\caption{Construction of the instance $\mathcal{I}'$}
    \label{fig:cn3V2}
\end{figure}	
	
    
    Consequently, we show that if $S^* \subseteq \R^*$ is a feasible assignment for $\mathcal{I}^*$, we can derive a feasible assignment for $\mathcal{I}'$ which then implies the feasibility of $\mathcal{I}$.
    To this end, we need to show that instance $\mathcal{I}'$ is feasible if and only if \cref{eq:FeasibilityLemma29} holds: 
    In case (i), instance $\mathcal{I}'$ has rooms of size $\{1, n \}$ and by construction $R_1 ' = R_a \geq n-1$ as well as $R_n' = R_{na}$. Consequently, we can apply \cref{eq:two_room_sizes} and $\mathcal{I}'$ is feasible if and only if $\nf' + \nm' \leq \sum_{r \in \R'} \croom'$ which is equivalent to \cref{eq:FeasibilityLemma29}.
    In case (ii), it holds $\croom' \in \{ 1, \dots, n \}$ for $r \in \R'$, and $R_i' = R_{ia} >0$ for $i = 1, \dots, n$. 
    And in case (iii), we get $\croom' \in \{b^0, b^1, \dots, b^n \} $ for $r \in \R'$, and $R_{b^i}' = R_{b^ia} \geq b-1$ for $i = 0, \dots, n$. Therefore, \cref{Lemma28} concludes in both cases that the instance $\mathcal{I}'$ is feasible if and only if \cref{eq:FeasibilityLemma29} is fulfilled. 
    Since \cref{eq:FeasibilityLemma29} is fulfilled if and only if $\mathcal{I}^*$ is feasible, it follows that $\mathcal{I}'$ and thus $\mathcal{I}$ are feasible.
\end{proof}

The new polynomial cases derived for the \ac{pra} can also be transferred to its related combinatorial problems. 
Moreover, the results indicate that for load factors $\ell \leq 1$, feasibility depends on the structural properties of the room configuration.
Room configurations that allow for more flexible partitions into male and female rooms can accommodate a wider range of gender compositions, while others are naturally more restrictive.
To assess the practical relevance of explicitly enforcing feasibility during the instance generation, we conduct a computational study showing that, despite this structural flexibility, instances still have infeasible days when the load factor becomes sufficiently high.


All instances are generated by varying the parameters room configuration, planning horizon, rate of female patients, and load factor. 
The full set of parameters and their values is summarized in \cref{tab:all_parameters}. 
Room configuration corresponds to the four bed configuration scenarios, which differ in the distribution of room sizes but have 30 beds each. The room configurations are chosen in such way that feasibility is explicitly not ensured as long as the number of patients does not exceed the number of beds. All instances have a load factor smaller or equal to one, ensuring that infeasibility arises solely from gender-based room restrictions.
For each parameter combination, 20 independent instances are generated without enforcing feasibility during instance generation. 
Feasibility is then assessed on a day-by-day basis using the proposed dynamic programming approach that checks whether the patients can be assigned to the rooms.

\begin{table}[ht]
\centering
\caption{Experimental parameters and their values considered in the computational study.}
\label{tab:all_parameters}
\begin{tabular}{@{}l l@{}}
\toprule
Parameter & Values \\
\midrule
Room configuration & Scenario 1: 10$\times$2-bed, 1$\times$4-bed, 1$\times$6-bed \\
                   & Scenario 2: 10$\times$3-bed \\
                   & Scenario 3: 2$\times$1-bed, 7$\times$4-bed \\
                   & Scenario 4: 6$\times$3-bed, 3$\times$4-bed \\
Planning horizon       & 30, 60 days \\
Rate of female patients & 0.1, 0.2, 0.3, 0.4, 0.5 \\
Load factor        & 0.90, 0.95, 0.98, 1.00 \\
\bottomrule
\end{tabular}
\end{table}

\cref{fig:boxplot_load_factor_planning_horizon} presents the proportion of infeasible days over the planning horizon for all parameter settings. 
A load factor of $\loadf \leq 90\%$ always yields feasible instances. For load factors approaching 100\%, however, the proportion of infeasible days increases sharply, reaching on average 50\% at full capacity. The results indicate that the proportion of infeasible days is primarily driven by the load factor.
We further observe that the length of the planning horizon has no impact on the daily infeasibility probability. However, longer time horizons naturally increase the probability of observing at least one infeasible day, making it more likely that the instance as a whole is infeasible.

\begin{figure}[h!]
    \centering
    \includegraphics[width=\linewidth]{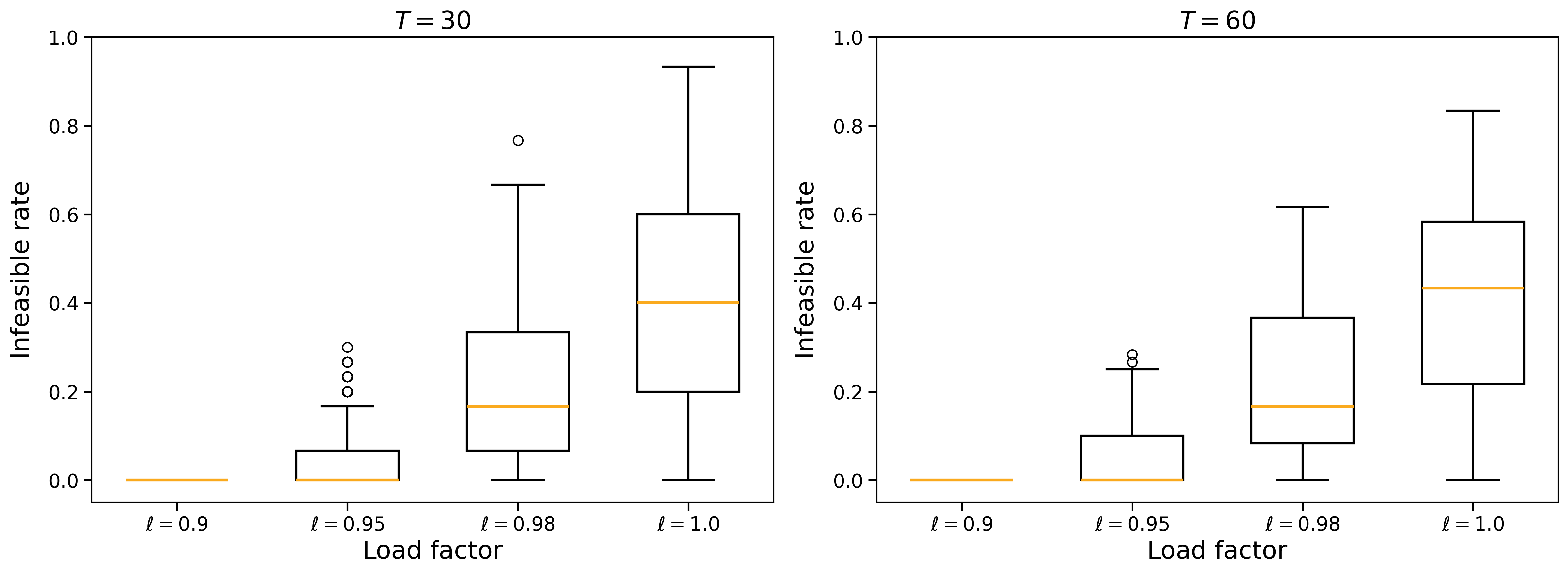}
    \caption{Proportion of infeasible days per planning horizon and per load factor.}
    \label{fig:boxplot_load_factor_planning_horizon}
\end{figure}


\section{Conclusion and future work}\label{sec:future}
Motivated by the need for realistic data in the evaluation of algorithms for patient-related problems and the lack thereof, we propose a new instance generator for the \ac{pra}. It is made publicly available and can be readily used to generate realistic and customizable instances.
Our approach is based on an extensive data analysis of a large real-life hospital dataset and provides distributions for key patient attributes, including age, length of stay, and length of register, as well as age-dependent characteristics. These distributions are made accessible through a user-friendly graphical interface that allows users to either select recommended configurations representing typical ward types or customize parameters to mimic specific hospital settings.

Beyond realism and flexibility, a central contribution of this generator is the optional enforcement of feasibility with respect to gender-separated rooms. In several countries, including Germany and the United Kingdom, shared hospital rooms are generally required to accommodate patients of the same sex. This can substantially affect the feasibility of patient-to-room assignments, particularly for long planning horizons and high load factors, as demonstrated by our computational study.
Building on existing complexity results for the \ac{pra} problem, we utilized the dynamic programming approach for the \acl{bkp} to check feasibility in pseudopolynomial time for arbitrary instances. This method is integrated into the generator, allowing users to explicitly enforce feasibility during instance generation when required. In addition, we identify further special cases for which the feasibility problem is solvable in polynomial time. 

While the generator is designed for the \ac{pra}, the generated instances can already be used as realistic input data for a broader class of patient-related planning problems in hospitals, potentially requiring minor problem-specific extensions.
Future work aims to integrate such extensions directly into the generator. Planned enhancements include incorporating temporal variations in patient arrivals, such as weekly patterns and seasonal effects. Further extensions comprise the simultaneous generation of multiple wards with shared pools of emergency patients and the inclusion of additional patient-related attributes, such as workload contributions for nurse-to-patient assignment.

\printbibliography

\begin{longtable}{llrrrr}
\caption{Computational Study Results Overview} \label{tab:computational_study_results_overview} \\
\toprule
idx & scenario & $\T$ & $p_{{\textnormal{{female}}}}$ & $\loadf$ & $\mu_{\text{infeas}} [d]$ \\
\midrule
\endfirsthead
\caption[]{Computational Study Results Overview} \\
\toprule
idx & scenario & $\T$ & $p_{{\textnormal{{female}}}}$ & $\loadf$ & $\mu_{\text{infeas}} [d]$ \\
\midrule
\endhead
\midrule
\multicolumn{6}{r}{Continued on next page} \\
\midrule
\endfoot
\bottomrule
\endlastfoot
0-19 & Scenario 1 (10x2, 1x4, 1x6) & 30 & 0.10 & 0.90 & 0.00 \\
20-39 & Scenario 1 (10x2, 1x4, 1x6) & 30 & 0.10 & 0.95 & 0.00 \\
\rowcolor[HTML]{f0f0f0} 40-59 & Scenario 1 (10x2, 1x4, 1x6) & 30 & 0.10 & 0.98 & 5.40 \\
\rowcolor[HTML]{f0f0f0} 60-79 & Scenario 1 (10x2, 1x4, 1x6) & 30 & 0.10 & 1.00 & 13.50 \\
80-99 & Scenario 1 (10x2, 1x4, 1x6) & 30 & 0.20 & 0.90 & 0.00 \\
100-119 & Scenario 1 (10x2, 1x4, 1x6) & 30 & 0.20 & 0.95 & 0.00 \\
\rowcolor[HTML]{f0f0f0} 120-139 & Scenario 1 (10x2, 1x4, 1x6) & 30 & 0.20 & 0.98 & 5.90 \\
\rowcolor[HTML]{f0f0f0} 140-159 & Scenario 1 (10x2, 1x4, 1x6) & 30 & 0.20 & 1.00 & 14.95 \\
160-179 & Scenario 1 (10x2, 1x4, 1x6) & 30 & 0.30 & 0.90 & 0.00 \\
180-199 & Scenario 1 (10x2, 1x4, 1x6) & 30 & 0.30 & 0.95 & 0.00 \\
\rowcolor[HTML]{f0f0f0} 200-219 & Scenario 1 (10x2, 1x4, 1x6) & 30 & 0.30 & 0.98 & 6.10 \\
\rowcolor[HTML]{f0f0f0} 220-239 & Scenario 1 (10x2, 1x4, 1x6) & 30 & 0.30 & 1.00 & 14.75 \\
240-259 & Scenario 1 (10x2, 1x4, 1x6) & 30 & 0.40 & 0.90 & 0.00 \\
260-279 & Scenario 1 (10x2, 1x4, 1x6) & 30 & 0.40 & 0.95 & 0.00 \\
\rowcolor[HTML]{f0f0f0} 280-299 & Scenario 1 (10x2, 1x4, 1x6) & 30 & 0.40 & 0.98 & 6.25 \\
\rowcolor[HTML]{f0f0f0} 300-319 & Scenario 1 (10x2, 1x4, 1x6) & 30 & 0.40 & 1.00 & 14.90 \\
320-339 & Scenario 1 (10x2, 1x4, 1x6) & 30 & 0.50 & 0.90 & 0.00 \\
340-359 & Scenario 1 (10x2, 1x4, 1x6) & 30 & 0.50 & 0.95 & 0.00 \\
\rowcolor[HTML]{f0f0f0} 360-379 & Scenario 1 (10x2, 1x4, 1x6) & 30 & 0.50 & 0.98 & 6.05 \\
\rowcolor[HTML]{f0f0f0} 380-399 & Scenario 1 (10x2, 1x4, 1x6) & 30 & 0.50 & 1.00 & 15.00 \\
400-419 & Scenario 1 (10x2, 1x4, 1x6) & 60 & 0.10 & 0.90 & 0.00 \\
420-439 & Scenario 1 (10x2, 1x4, 1x6) & 60 & 0.10 & 0.95 & 0.00 \\
\rowcolor[HTML]{f0f0f0} 440-459 & Scenario 1 (10x2, 1x4, 1x6) & 60 & 0.10 & 0.98 & 12.20 \\
\rowcolor[HTML]{f0f0f0} 460-479 & Scenario 1 (10x2, 1x4, 1x6) & 60 & 0.10 & 1.00 & 30.20 \\
480-499 & Scenario 1 (10x2, 1x4, 1x6) & 60 & 0.20 & 0.90 & 0.00 \\
500-519 & Scenario 1 (10x2, 1x4, 1x6) & 60 & 0.20 & 0.95 & 0.00 \\
\rowcolor[HTML]{f0f0f0} 520-539 & Scenario 1 (10x2, 1x4, 1x6) & 60 & 0.20 & 0.98 & 11.45 \\
\rowcolor[HTML]{f0f0f0} 540-559 & Scenario 1 (10x2, 1x4, 1x6) & 60 & 0.20 & 1.00 & 30.40 \\
560-579 & Scenario 1 (10x2, 1x4, 1x6) & 60 & 0.30 & 0.90 & 0.00 \\
580-599 & Scenario 1 (10x2, 1x4, 1x6) & 60 & 0.30 & 0.95 & 0.00 \\
\rowcolor[HTML]{f0f0f0} 600-619 & Scenario 1 (10x2, 1x4, 1x6) & 60 & 0.30 & 0.98 & 11.50 \\
\rowcolor[HTML]{f0f0f0} 620-639 & Scenario 1 (10x2, 1x4, 1x6) & 60 & 0.30 & 1.00 & 29.15 \\
640-659 & Scenario 1 (10x2, 1x4, 1x6) & 60 & 0.40 & 0.90 & 0.00 \\
660-679 & Scenario 1 (10x2, 1x4, 1x6) & 60 & 0.40 & 0.95 & 0.00 \\
\rowcolor[HTML]{f0f0f0} 680-699 & Scenario 1 (10x2, 1x4, 1x6) & 60 & 0.40 & 0.98 & 12.15 \\
\rowcolor[HTML]{f0f0f0} 700-719 & Scenario 1 (10x2, 1x4, 1x6) & 60 & 0.40 & 1.00 & 28.30 \\
720-739 & Scenario 1 (10x2, 1x4, 1x6) & 60 & 0.50 & 0.90 & 0.00 \\
740-759 & Scenario 1 (10x2, 1x4, 1x6) & 60 & 0.50 & 0.95 & 0.00 \\
\rowcolor[HTML]{f0f0f0} 760-779 & Scenario 1 (10x2, 1x4, 1x6) & 60 & 0.50 & 0.98 & 11.45 \\
\rowcolor[HTML]{f0f0f0} 780-799 & Scenario 1 (10x2, 1x4, 1x6) & 60 & 0.50 & 1.00 & 30.60 \\
800-819 & Scenario 2 (10x3) & 30 & 0.10 & 0.90 & 0.00 \\
\rowcolor[HTML]{f0f0f0} 820-839 & Scenario 2 (10x3) & 30 & 0.10 & 0.95 & 5.50 \\
\rowcolor[HTML]{f0f0f0} 840-859 & Scenario 2 (10x3) & 30 & 0.10 & 0.98 & 14.70 \\
\rowcolor[HTML]{f0f0f0} 860-879 & Scenario 2 (10x3) & 30 & 0.10 & 1.00 & 21.20 \\
880-899 & Scenario 2 (10x3) & 30 & 0.20 & 0.90 & 0.00 \\
\rowcolor[HTML]{f0f0f0} 900-919 & Scenario 2 (10x3) & 30 & 0.20 & 0.95 & 5.20 \\
\rowcolor[HTML]{f0f0f0} 920-939 & Scenario 2 (10x3) & 30 & 0.20 & 0.98 & 13.40 \\
\rowcolor[HTML]{f0f0f0} 940-959 & Scenario 2 (10x3) & 30 & 0.20 & 1.00 & 20.65 \\
960-979 & Scenario 2 (10x3) & 30 & 0.30 & 0.90 & 0.00 \\
\rowcolor[HTML]{f0f0f0} 980-999 & Scenario 2 (10x3) & 30 & 0.30 & 0.95 & 4.45 \\
\rowcolor[HTML]{f0f0f0} 1000-1019 & Scenario 2 (10x3) & 30 & 0.30 & 0.98 & 13.40 \\
\rowcolor[HTML]{f0f0f0} 1020-1039 & Scenario 2 (10x3) & 30 & 0.30 & 1.00 & 19.70 \\
1040-1059 & Scenario 2 (10x3) & 30 & 0.40 & 0.90 & 0.00 \\
\rowcolor[HTML]{f0f0f0} 1060-1079 & Scenario 2 (10x3) & 30 & 0.40 & 0.95 & 4.80 \\
\rowcolor[HTML]{f0f0f0} 1080-1099 & Scenario 2 (10x3) & 30 & 0.40 & 0.98 & 13.90 \\
\rowcolor[HTML]{f0f0f0} 1100-1119 & Scenario 2 (10x3) & 30 & 0.40 & 1.00 & 19.05 \\
1120-1139 & Scenario 2 (10x3) & 30 & 0.50 & 0.90 & 0.00 \\
\rowcolor[HTML]{f0f0f0} 1140-1159 & Scenario 2 (10x3) & 30 & 0.50 & 0.95 & 5.25 \\
\rowcolor[HTML]{f0f0f0} 1160-1179 & Scenario 2 (10x3) & 30 & 0.50 & 0.98 & 13.60 \\
\rowcolor[HTML]{f0f0f0} 1180-1199 & Scenario 2 (10x3) & 30 & 0.50 & 1.00 & 19.40 \\
1200-1219 & Scenario 2 (10x3) & 60 & 0.10 & 0.90 & 0.00 \\
\rowcolor[HTML]{f0f0f0} 1220-1239 & Scenario 2 (10x3) & 60 & 0.10 & 0.95 & 10.30 \\
\rowcolor[HTML]{f0f0f0} 1240-1259 & Scenario 2 (10x3) & 60 & 0.10 & 0.98 & 28.05 \\
\rowcolor[HTML]{f0f0f0} 1260-1279 & Scenario 2 (10x3) & 60 & 0.10 & 1.00 & 39.70 \\
1280-1299 & Scenario 2 (10x3) & 60 & 0.20 & 0.90 & 0.00 \\
\rowcolor[HTML]{f0f0f0} 1300-1319 & Scenario 2 (10x3) & 60 & 0.20 & 0.95 & 9.65 \\
\rowcolor[HTML]{f0f0f0} 1320-1339 & Scenario 2 (10x3) & 60 & 0.20 & 0.98 & 28.20 \\
\rowcolor[HTML]{f0f0f0} 1340-1359 & Scenario 2 (10x3) & 60 & 0.20 & 1.00 & 39.50 \\
1360-1379 & Scenario 2 (10x3) & 60 & 0.30 & 0.90 & 0.00 \\
\rowcolor[HTML]{f0f0f0} 1380-1399 & Scenario 2 (10x3) & 60 & 0.30 & 0.95 & 9.55 \\
\rowcolor[HTML]{f0f0f0} 1400-1419 & Scenario 2 (10x3) & 60 & 0.30 & 0.98 & 28.45 \\
\rowcolor[HTML]{f0f0f0} 1420-1439 & Scenario 2 (10x3) & 60 & 0.30 & 1.00 & 39.75 \\
1440-1459 & Scenario 2 (10x3) & 60 & 0.40 & 0.90 & 0.00 \\
\rowcolor[HTML]{f0f0f0} 1460-1479 & Scenario 2 (10x3) & 60 & 0.40 & 0.95 & 11.10 \\
\rowcolor[HTML]{f0f0f0} 1480-1499 & Scenario 2 (10x3) & 60 & 0.40 & 0.98 & 28.25 \\
\rowcolor[HTML]{f0f0f0} 1500-1519 & Scenario 2 (10x3) & 60 & 0.40 & 1.00 & 39.45 \\
1520-1539 & Scenario 2 (10x3) & 60 & 0.50 & 0.90 & 0.00 \\
\rowcolor[HTML]{f0f0f0} 1540-1559 & Scenario 2 (10x3) & 60 & 0.50 & 0.95 & 10.65 \\
\rowcolor[HTML]{f0f0f0} 1560-1579 & Scenario 2 (10x3) & 60 & 0.50 & 0.98 & 27.65 \\
\rowcolor[HTML]{f0f0f0} 1580-1599 & Scenario 2 (10x3) & 60 & 0.50 & 1.00 & 40.85 \\
1600-1619 & Scenario 3 (2x1, 7x4) & 30 & 0.10 & 0.90 & 0.00 \\
1620-1639 & Scenario 3 (2x1, 7x4) & 30 & 0.10 & 0.95 & 0.00 \\
\rowcolor[HTML]{f0f0f0} 1640-1659 & Scenario 3 (2x1, 7x4) & 30 & 0.10 & 0.98 & 3.15 \\
\rowcolor[HTML]{f0f0f0} 1660-1679 & Scenario 3 (2x1, 7x4) & 30 & 0.10 & 1.00 & 7.70 \\
1680-1699 & Scenario 3 (2x1, 7x4) & 30 & 0.20 & 0.90 & 0.00 \\
1700-1719 & Scenario 3 (2x1, 7x4) & 30 & 0.20 & 0.95 & 0.00 \\
\rowcolor[HTML]{f0f0f0} 1720-1739 & Scenario 3 (2x1, 7x4) & 30 & 0.20 & 0.98 & 2.85 \\
\rowcolor[HTML]{f0f0f0} 1740-1759 & Scenario 3 (2x1, 7x4) & 30 & 0.20 & 1.00 & 8.05 \\
1760-1779 & Scenario 3 (2x1, 7x4) & 30 & 0.30 & 0.90 & 0.00 \\
1780-1799 & Scenario 3 (2x1, 7x4) & 30 & 0.30 & 0.95 & 0.00 \\
\rowcolor[HTML]{f0f0f0} 1800-1819 & Scenario 3 (2x1, 7x4) & 30 & 0.30 & 0.98 & 2.85 \\
\rowcolor[HTML]{f0f0f0} 1820-1839 & Scenario 3 (2x1, 7x4) & 30 & 0.30 & 1.00 & 7.40 \\
1840-1859 & Scenario 3 (2x1, 7x4) & 30 & 0.40 & 0.90 & 0.00 \\
1860-1879 & Scenario 3 (2x1, 7x4) & 30 & 0.40 & 0.95 & 0.00 \\
\rowcolor[HTML]{f0f0f0} 1880-1899 & Scenario 3 (2x1, 7x4) & 30 & 0.40 & 0.98 & 3.10 \\
\rowcolor[HTML]{f0f0f0} 1900-1919 & Scenario 3 (2x1, 7x4) & 30 & 0.40 & 1.00 & 8.10 \\
1920-1939 & Scenario 3 (2x1, 7x4) & 30 & 0.50 & 0.90 & 0.00 \\
1940-1959 & Scenario 3 (2x1, 7x4) & 30 & 0.50 & 0.95 & 0.00 \\
\rowcolor[HTML]{f0f0f0} 1960-1979 & Scenario 3 (2x1, 7x4) & 30 & 0.50 & 0.98 & 2.60 \\
\rowcolor[HTML]{f0f0f0} 1980-1999 & Scenario 3 (2x1, 7x4) & 30 & 0.50 & 1.00 & 6.35 \\
2000-2019 & Scenario 3 (2x1, 7x4) & 60 & 0.10 & 0.90 & 0.00 \\
2020-2039 & Scenario 3 (2x1, 7x4) & 60 & 0.10 & 0.95 & 0.00 \\
\rowcolor[HTML]{f0f0f0} 2040-2059 & Scenario 3 (2x1, 7x4) & 60 & 0.10 & 0.98 & 5.70 \\
\rowcolor[HTML]{f0f0f0} 2060-2079 & Scenario 3 (2x1, 7x4) & 60 & 0.10 & 1.00 & 13.25 \\
2080-2099 & Scenario 3 (2x1, 7x4) & 60 & 0.20 & 0.90 & 0.00 \\
2100-2119 & Scenario 3 (2x1, 7x4) & 60 & 0.20 & 0.95 & 0.00 \\
\rowcolor[HTML]{f0f0f0} 2120-2139 & Scenario 3 (2x1, 7x4) & 60 & 0.20 & 0.98 & 4.90 \\
\rowcolor[HTML]{f0f0f0} 2140-2159 & Scenario 3 (2x1, 7x4) & 60 & 0.20 & 1.00 & 15.95 \\
2160-2179 & Scenario 3 (2x1, 7x4) & 60 & 0.30 & 0.90 & 0.00 \\
2180-2199 & Scenario 3 (2x1, 7x4) & 60 & 0.30 & 0.95 & 0.00 \\
\rowcolor[HTML]{f0f0f0} 2200-2219 & Scenario 3 (2x1, 7x4) & 60 & 0.30 & 0.98 & 6.40 \\
\rowcolor[HTML]{f0f0f0} 2220-2239 & Scenario 3 (2x1, 7x4) & 60 & 0.30 & 1.00 & 14.60 \\
2240-2259 & Scenario 3 (2x1, 7x4) & 60 & 0.40 & 0.90 & 0.00 \\
2260-2279 & Scenario 3 (2x1, 7x4) & 60 & 0.40 & 0.95 & 0.00 \\
\rowcolor[HTML]{f0f0f0} 2280-2299 & Scenario 3 (2x1, 7x4) & 60 & 0.40 & 0.98 & 5.75 \\
\rowcolor[HTML]{f0f0f0} 2300-2319 & Scenario 3 (2x1, 7x4) & 60 & 0.40 & 1.00 & 14.80 \\
2320-2339 & Scenario 3 (2x1, 7x4) & 60 & 0.50 & 0.90 & 0.00 \\
2340-2359 & Scenario 3 (2x1, 7x4) & 60 & 0.50 & 0.95 & 0.00 \\
\rowcolor[HTML]{f0f0f0} 2360-2379 & Scenario 3 (2x1, 7x4) & 60 & 0.50 & 0.98 & 6.65 \\
\rowcolor[HTML]{f0f0f0} 2380-2399 & Scenario 3 (2x1, 7x4) & 60 & 0.50 & 1.00 & 15.65 \\
2400-2419 & Scenario 4 (6x3, 3x4) & 30 & 0.10 & 0.90 & 0.00 \\
\rowcolor[HTML]{f0f0f0} 2420-2439 & Scenario 4 (6x3, 3x4) & 30 & 0.10 & 0.95 & 1.80 \\
\rowcolor[HTML]{f0f0f0} 2440-2459 & Scenario 4 (6x3, 3x4) & 30 & 0.10 & 0.98 & 9.75 \\
\rowcolor[HTML]{f0f0f0} 2460-2479 & Scenario 4 (6x3, 3x4) & 30 & 0.10 & 1.00 & 14.60 \\
2480-2499 & Scenario 4 (6x3, 3x4) & 30 & 0.20 & 0.90 & 0.00 \\
\rowcolor[HTML]{f0f0f0} 2500-2519 & Scenario 4 (6x3, 3x4) & 30 & 0.20 & 0.95 & 0.20 \\
\rowcolor[HTML]{f0f0f0} 2520-2539 & Scenario 4 (6x3, 3x4) & 30 & 0.20 & 0.98 & 3.15 \\
\rowcolor[HTML]{f0f0f0} 2540-2559 & Scenario 4 (6x3, 3x4) & 30 & 0.20 & 1.00 & 6.85 \\
2560-2579 & Scenario 4 (6x3, 3x4) & 30 & 0.30 & 0.90 & 0.00 \\
2580-2599 & Scenario 4 (6x3, 3x4) & 30 & 0.30 & 0.95 & 0.00 \\
\rowcolor[HTML]{f0f0f0} 2600-2619 & Scenario 4 (6x3, 3x4) & 30 & 0.30 & 0.98 & 0.35 \\
\rowcolor[HTML]{f0f0f0} 2620-2639 & Scenario 4 (6x3, 3x4) & 30 & 0.30 & 1.00 & 1.65 \\
2640-2659 & Scenario 4 (6x3, 3x4) & 30 & 0.40 & 0.90 & 0.00 \\
2660-2679 & Scenario 4 (6x3, 3x4) & 30 & 0.40 & 0.95 & 0.00 \\
2680-2699 & Scenario 4 (6x3, 3x4) & 30 & 0.40 & 0.98 & 0.00 \\
\rowcolor[HTML]{f0f0f0} 2700-2719 & Scenario 4 (6x3, 3x4) & 30 & 0.40 & 1.00 & 0.20 \\
2720-2739 & Scenario 4 (6x3, 3x4) & 30 & 0.50 & 0.90 & 0.00 \\
2740-2759 & Scenario 4 (6x3, 3x4) & 30 & 0.50 & 0.95 & 0.00 \\
2760-2779 & Scenario 4 (6x3, 3x4) & 30 & 0.50 & 0.98 & 0.00 \\
2780-2799 & Scenario 4 (6x3, 3x4) & 30 & 0.50 & 1.00 & 0.00 \\
2800-2819 & Scenario 4 (6x3, 3x4) & 60 & 0.10 & 0.90 & 0.00 \\
\rowcolor[HTML]{f0f0f0} 2820-2839 & Scenario 4 (6x3, 3x4) & 60 & 0.10 & 0.95 & 5.75 \\
\rowcolor[HTML]{f0f0f0} 2840-2859 & Scenario 4 (6x3, 3x4) & 60 & 0.10 & 0.98 & 16.95 \\
\rowcolor[HTML]{f0f0f0} 2860-2879 & Scenario 4 (6x3, 3x4) & 60 & 0.10 & 1.00 & 28.80 \\
2880-2899 & Scenario 4 (6x3, 3x4) & 60 & 0.20 & 0.90 & 0.00 \\
\rowcolor[HTML]{f0f0f0} 2900-2919 & Scenario 4 (6x3, 3x4) & 60 & 0.20 & 0.95 & 0.50 \\
\rowcolor[HTML]{f0f0f0} 2920-2939 & Scenario 4 (6x3, 3x4) & 60 & 0.20 & 0.98 & 5.20 \\
\rowcolor[HTML]{f0f0f0} 2940-2959 & Scenario 4 (6x3, 3x4) & 60 & 0.20 & 1.00 & 12.75 \\
2960-2979 & Scenario 4 (6x3, 3x4) & 60 & 0.30 & 0.90 & 0.00 \\
\rowcolor[HTML]{f0f0f0} 2980-2999 & Scenario 4 (6x3, 3x4) & 60 & 0.30 & 0.95 & 0.05 \\
\rowcolor[HTML]{f0f0f0} 3000-3019 & Scenario 4 (6x3, 3x4) & 60 & 0.30 & 0.98 & 1.05 \\
\rowcolor[HTML]{f0f0f0} 3020-3039 & Scenario 4 (6x3, 3x4) & 60 & 0.30 & 1.00 & 1.40 \\
3040-3059 & Scenario 4 (6x3, 3x4) & 60 & 0.40 & 0.90 & 0.00 \\
3060-3079 & Scenario 4 (6x3, 3x4) & 60 & 0.40 & 0.95 & 0.00 \\
\rowcolor[HTML]{f0f0f0} 3080-3099 & Scenario 4 (6x3, 3x4) & 60 & 0.40 & 0.98 & 0.25 \\
\rowcolor[HTML]{f0f0f0} 3100-3119 & Scenario 4 (6x3, 3x4) & 60 & 0.40 & 1.00 & 0.25 \\
3120-3139 & Scenario 4 (6x3, 3x4) & 60 & 0.50 & 0.90 & 0.00 \\
3140-3159 & Scenario 4 (6x3, 3x4) & 60 & 0.50 & 0.95 & 0.00 \\
3160-3179 & Scenario 4 (6x3, 3x4) & 60 & 0.50 & 0.98 & 0.00 \\
3180-3199 & Scenario 4 (6x3, 3x4) & 60 & 0.50 & 1.00 & 0.00 \\
\end{longtable}

\end{document}